\documentclass[11pt]{article}    
\usepackage{amsmath, amssymb, amsthm,pdfsync}
\usepackage{fullpage}
\usepackage{verbatim}
\usepackage{graphicx}
\usepackage{subfig, caption}
\usepackage{cancel}
\usepackage{enumerate}
\usepackage{wrapfig}
\usepackage{mathtools}
\usepackage{xcolor}
\usepackage{epigraph}
\usepackage{cleveref}
\usepackage{overpic}
\usepackage{esint}
\usepackage{dsfont}
\usepackage{color}
\usepackage{titling}

\newtheorem{theorem}{Theorem}
\newtheorem{lemma}[theorem]{Lemma}
\newtheorem{lem}[theorem]{Lemma}
\newtheorem{prop}[theorem]{Proposition}

\newtheorem{assumption}[theorem]{Assumption}

\newcommand{\vertiii}[1]{{\left\vert\kern-0.25ex\left\vert\kern-0.25ex\left\vert #1 
    \right\vert\kern-0.25ex\right\vert\kern-0.25ex\right\vert}}
\newcommand{\e}{\varepsilon}

\newcommand{\R}{\mathbb{R}}

\newcommand{\1}{\mathds{1}}

\DeclareMathOperator{\dist}{dist}

\DeclareMathOperator{\Int}{Int}

\numberwithin{equation}{section}
\numberwithin{theorem}{section}

{\setlength{\parindent}{0cm}

\Crefname{assumption}{Assumption}{Assumptions}
\Crefname{theorem}{Theorem}{Theorems}
\Crefname{lem}{Lemma}{Lemmas}
\Crefname{cor}{Corollary}{Corollaries}
\Crefname{prop}{Proposition}{Propositions}
\Crefname{theorem}{Theorem}{Theorems}
\Crefname{conjecture}{Conjecture}{Conjectures}

\begin{document}
\title{Super-linear propagation for a general, local cane toads model}

%

\setlength\thanksmarkwidth{.5em}
\setlength\thanksmargin{-\thanksmarkwidth}

\author{
	Christopher Henderson\thanks{Corresponding author, Department of Mathematics, The University of Chicago, 5734 S.~University Avenue, Chicago, IL 60637, E-mail: \texttt{henderson@math.uchicago.edu}},
	Beno\^{i}t Perthame\thanks{Sorbonne Universit\'e, CNRS, Univ.~Paris Diderot, INRIA, Laboratoire Jacques-Louis Lions, F-75005 Paris, France,
	\texttt{benoit.perthame@sorbonne-universite.fr}},
	Panagiotis E.~Souganidis\thanks{Department of Mathematics, The University of Chicago, 5734 S.~University Avenue, Chicago, IL 60637, E-mail: \texttt{souganidis@math.uchicago.edu}}
}

\maketitle
\begin{abstract}
\noindent We investigate a general, local version of the cane toads equation, 
which models the spread of a population structured by unbounded motility.  We use the thin-front limit approach  of Evans and Souganidis in [{\em Indiana Univ.~Math.~J.}, 1989] to obtain a characterization of the propagation in terms of both the linearized equation and a geometric front equation.  In particular, we reduce the task of understanding the precise location of the front for a large class of equations to analyzing a much smaller class of Hamilton-Jacobi equations.  We are then able to give an explicit formula for the front location in physical space.  One advantage of our approach is that we do not use the explicit trajectories along which the population spreads, which was a basis of previous work.  Our result allows for large oscillations in the motility.
\end{abstract}
%


\section{Introduction and Main Results}\label{sec:results}

The cane toads equation models the spread of a population where the motility of the individuals is not constant.  Its name comes from the cane toads in Australia whose invasion has been the subject of intense biological interest in recent years; see for example Phillips et.~al.~\cite{phillips2006invasion} and Shine et.~al.~\cite{Shine}. This phenomenon has been observed more widely, for example, the expansion of bush crickets in Great Britain, see Thomas et.~al.~\cite{Thomas}.  The mathematical model presented here has its roots in the work of Arnold, Desvillettes, and Prevost~\cite{ArnoldDesvillettesPrevost}, Champagnat and M\'el\'eard~\cite{ChampagnatMeleard}, and Benichou et.~al.~\cite{BenichouEtAl}.

\smallskip

The equation that we study is a general, local version of the cane toads equation.   In what follows $t$ represents time, $x$ physical space, and $\theta$ the genetic trait of motility.  The equation is 
\begin{equation}\label{eq:cane_toads_unscaled}
	\begin{cases}
		u_t = D(\theta) u_{xx} + u_{\theta\theta} + u(1-u) \quad  \text{in}  \quad  \R\times\R^+\times\R^+,\\
		u_\theta(x,0,t) = 0  \quad  \text{in}  \quad  \R\times \R^+,
	\end{cases}
\end{equation}
where $\R^+:=(0,\infty),$ with the diffusion coefficient $D: \R^+ \to \R^+$, a continuous function satisfying:
\begin{assumption}\label{assumption:D}
Let $\overline D^\epsilon(\theta) := D(\theta/\epsilon)/D(1/\epsilon)$. There exists $\overline D:\R^+ \to \R^+$ such that, locally uniformly in  $ \R^+$, $\lim_{\e\to0} \overline D^\e(\theta) = \overline D(\theta)$,
and $\lim_{\theta\to\infty} \overline D(\theta) = \lim_{\theta\to\infty} D(\theta) = \infty$.
\end{assumption}

 In fact, the convergence in \Cref{assumption:D} implies immediately that $\overline D(\theta) = \theta^p$ for some $p>0$.  Indeed, $\overline D$ is Borel measurable, as it is the limit of Borel measurable functions, and it is multiplicative because, for any $\theta_1,\theta_2\in \R^+$,
\[
	\overline D(\theta_1 \theta_2)
		= \lim_{\e \to 0} \frac{D(\theta_1\theta_2/\e)}{D(1/\e)}
		= \lim_{\e \to 0} \frac{D(\theta_1/\e)}{D(1/\e)}\frac{D(\theta_1\theta_2/\e)}{D(\theta_1/\e)}
		= \overline D(\theta_1) \overline D(\theta_2).
\]
It is well-known that these two properties imply that $\overline D$ is a homogeneous polynomial.

The case most often considered is, up to translation in $\theta$, $D(\theta) = \theta + \underline\theta$ for some $\underline\theta>0$, whence $\overline D(\theta) = \theta$, see~\cite{BerestyckiMouhotRaoul,BCMetal,BHR_acceleration}.  A non-trivial example is
\begin{equation}\label{eq:D_example}
	D(\theta) = \theta (1 + \log(\theta+1) + \sin(\theta)),
\end{equation}
which, despite having arbitrarily large oscillations, nevertheless satisfies $\overline D(\theta)= \theta$.

The biological motivation for studying the equation~\eqref{eq:cane_toads_unscaled} in greater generality is one of modelling: in the present work, we see various propagation rates depending on the asymptotics of $D$ and this may be used to fit the model to the phenomenon being studied.  Indeed, even for data arising from the cane toads invasion in Australia there is some uncertainty over the propagation rate (see~\cite[Table 1]{urban2008toad}).  Further,  there is no reason that the $O(t^{3/2})$ propagation, which is associated to the choice $D(\theta) = \theta + \underline\theta$, should hold for all species with increasing motility.  Hence, it is important to have general models that can be tailored to each species.
\smallskip

We are interested in the long time, large space and motility limit.  Fix a small parameter $\e \in (0,1)$.  Thinking of the time scale as $\epsilon^{-1}$, the scaled function
\[
	u^\epsilon(x,\theta,t) = u\left(\frac{x \sqrt{D(1/\epsilon)}}{\epsilon} , \frac{\theta}{\epsilon}, \frac{t}{\epsilon}\right)
\]
satisfies
\begin{equation}\label{eq:cane_toads}
	\begin{cases}
		u^\epsilon_t = \epsilon \overline D^\epsilon(\theta) u^\epsilon_{xx} + \epsilon u^\epsilon_{\theta\theta} + \frac{1}{\epsilon} u^\epsilon(1-u^\epsilon) \quad &\text{ in } \quad \R\times\R^+\times\R^+,\\
		u^\epsilon_\theta(x,0,t) = 0 \quad &\text{ on } \quad  \R \times \R^+,
	\end{cases}
\end{equation}
which  
we supplement with the initial condition
\begin{equation}\label{eq:initial_condition}
	u^\epsilon(x,\theta,0) = u_0(x,\theta)\quad \text{ with } \quad 0 \leq u_0 \leq 1,
\end{equation}
where $u_0$ satisfies the following assumption:
\begin{assumption}\label{assumption:u_0}
The initial data $u_0$ is continuous and supported on $G_0$, a $C^3$, open, non-empty, convex subset of $\R\times [0,\infty)$ such that $G_0 \cap (\R^+ \times [0,\infty))$ is bounded; 
 that is, there exist  $\overline\theta>0$ and $x_r \in \R$ such that $G_0 \subset (-\infty,x_r)\times[0,\overline\theta).$
\end{assumption}
The assumption that $0 \leq u_0 \leq 1$ yields, by the maximum principle, $0 \leq u^\e \leq 1$.  We note that the restrictions that $u_0$ is continuous and that $u_0 \leq 1$ are made for simplicity.  Indeed, we use continuity only to guarantee that $\min_Q u_0 > 0$ for any compact set $Q\subset G_0$ and the necessary modifications to handle the case when $u_0 \not \leq 1$ may be found in~\cite[Lemma~1.2 and (2.5)]{EvansSouganidis}.  In addition, we note that our approach could be generalized to higher spatial dimensions with no added difficulty, only additional notation.  For simplicity, we present here only the one-dimensional model.
%
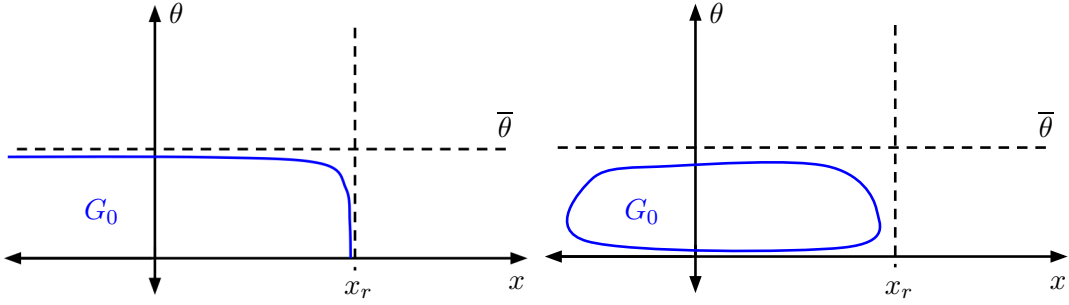
\begin{figure}
\begin{center}
\begin{overpic}[scale = .33]
		{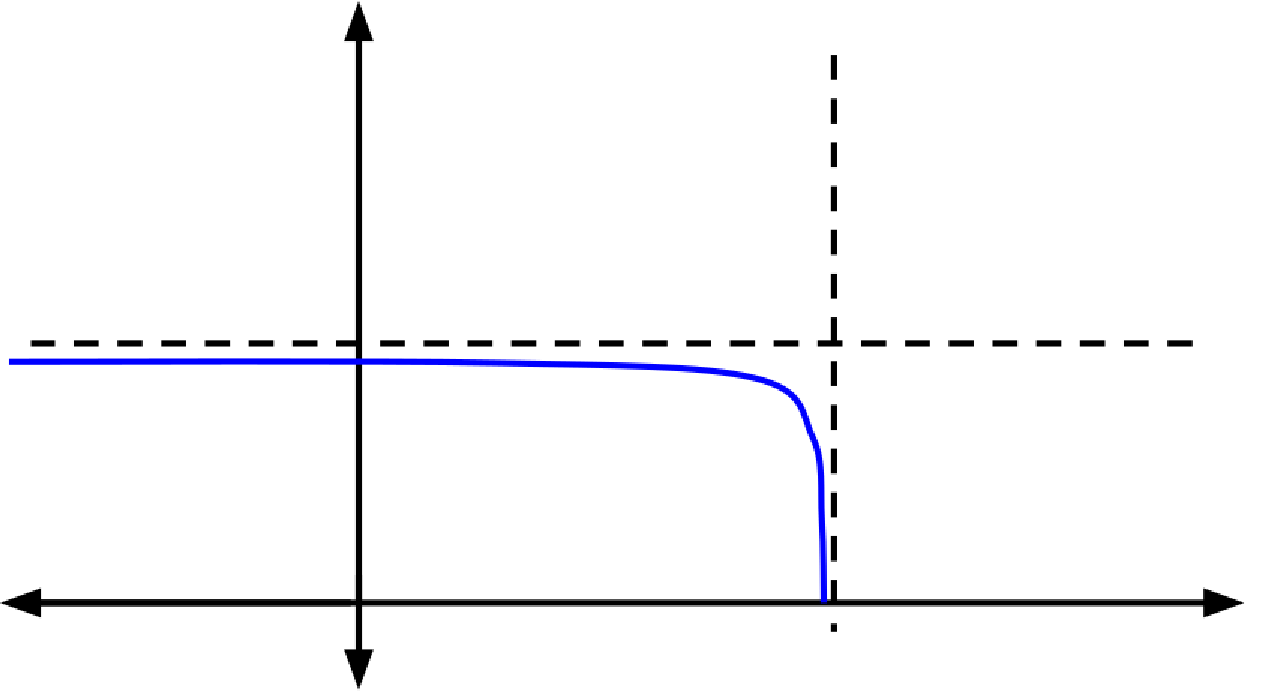}
	\put (95,2){$x$}	
	\put (31,52){$\theta$}	
	\put (64,1){$x_r$}
	\put (93, 30.5){$\overline \theta$}
	\put (15,15) {\color{blue} $G_0$}
\end{overpic}
\begin{overpic}[scale = .33]
		{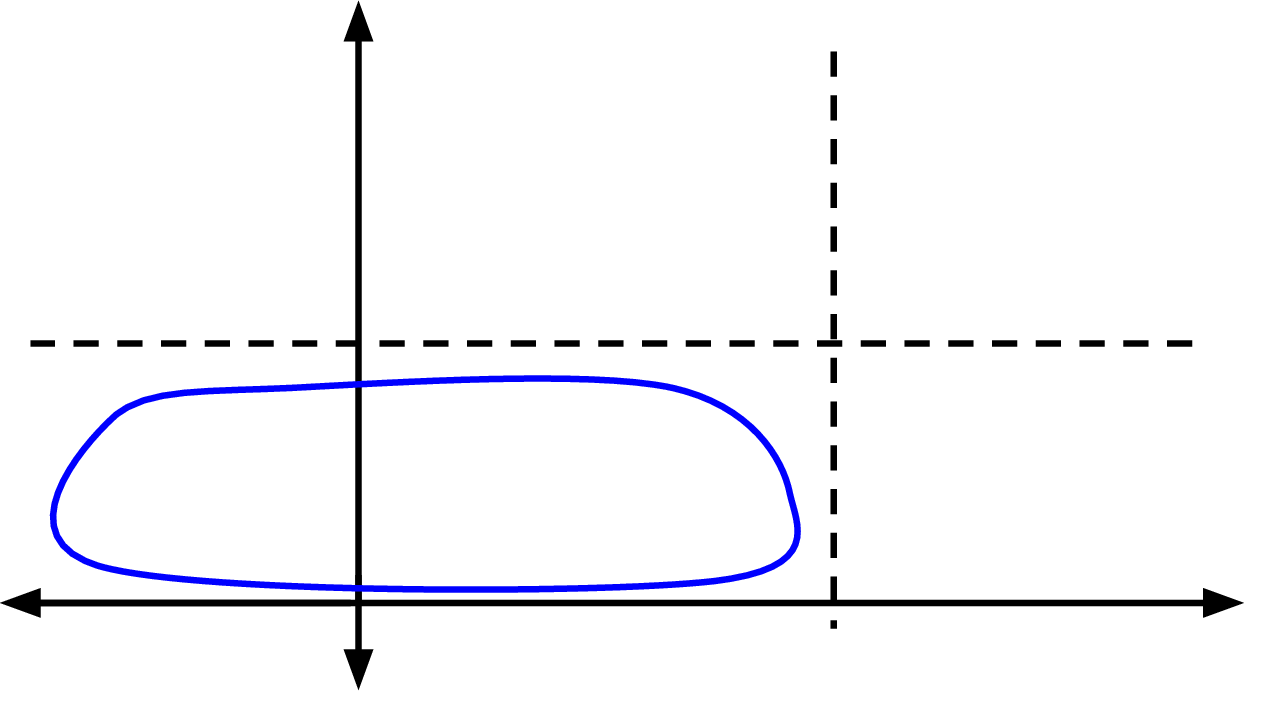}
	\put (95,2){$x$}	
	\put (31,52){$\theta$}	
	\put (64,1){$x_r$}
	\put (93, 30.5){$\overline \theta$}
	\put (15,15) {\color{blue} $G_0$}\end{overpic}
\caption{Two representative examples of $G_0$.}
\end{center}
\end{figure}

\smallskip
To study the behavior of $u^\epsilon$ as $\epsilon$ tends to zero, following Evans and Souganidis~\cite{EvansSouganidis}, we make the transformation $v^\epsilon = - \epsilon \log u^\epsilon$.   This is referred to as the Hopf-Cole transform and is standard in the literature (see also~\cite{BarlesEvansSouganidis,MajdaSouganidis} for applications to reaction-diffusion problems, \cite{Fleming} and references therein for early applications to large deviations and other problems, and~\cite{Cole,Hopf} for the original introduction of the transformation by Cole and Hopf).   Since $0 < u^\e < 1$ in $\R\times \R^+ \times \R^+$, then $0 < v^\e < +\infty$ in $\R\times\R^+ \times \R^+$. Also, $v^\epsilon$ satisfies
\begin{equation}\label{eq:phase}
\begin{cases}
		v^\epsilon_t - \epsilon \overline D^\epsilon(\theta) v^\epsilon_{xx} - \epsilon v^\epsilon_{\theta\theta} + \overline D^\epsilon(\theta)|v^\epsilon_x|^2 + |v^\epsilon_\theta|^2 + 1 - e^{-v^\epsilon/\epsilon} = 0  \quad  &\text{ in }  \quad  \R\times\R^+\times\R^+\\[1mm]
		v^\epsilon_\theta(x, 0,t)=0 \quad &\text{ on } \quad  \R\times \R^+,
		\end{cases}
\end{equation}
with initial conditions $v^\epsilon(x,\theta,0) = v^\epsilon_0(x,\theta)$, where
\[
	v^\epsilon_0
	= \begin{cases}
		-\epsilon \log u_0 \quad &\text{ in }  \quad  G_0,\\[1mm]
		\infty  \quad &\text{ in } \quad  \overline G_0^c.
	\end{cases}
\]
From the above, we see that, formally,~\eqref{eq:phase} converges, when $v>0$,  to $v_t + \overline D |v_x|^2 + |v_\theta|^2 + 1 = 0.$
Indeed, the following lemma shows this to be the case.
\begin{prop}\label{prop:phi_convergence}
Suppose that  \Cref{assumption:D} and \Cref{assumption:u_0} hold.  Then, as $\epsilon$ tends to $0$ and locally uniformly in $\R\times  [0,\infty)\times\R^+,$  the $v^\epsilon$'s converge to $I$, 
which is the unique solution of
\begin{equation}\label{eq:I}
	\begin{cases}
		\min\left\{ I_t + \overline D(\theta) |I_x|^2 + |I_\theta|^2 + 1, I\right\} = 0 \quad &\text{ in } \quad \R \times \R^+\times \R^+,\\
		\max\left\{ - I_\theta, \min\left\{I_t + |I_\theta|^2 + 1, I\right\}\right\} \geq 0  \quad &\text{ on }  \quad \R\times \{0\}\times \R^+,\\
		\min\left\{ - I_\theta, \min\left\{I_t + |I_\theta|^2 + 1, I\right\}\right\} \leq 0 \quad&\text{ on }  \quad \R\times \{0\}\times \R^+,
	\end{cases}
\end{equation}
and
\begin{equation}\label{eq:I_0_intro}
	I(x,\theta,0) = \begin{cases}
		\infty  \quad &\text{ in }\quad \overline G_0^c,\\[1mm]
		0 \quad &\text{ in }\quad G_0.
	\end{cases}
\end{equation}
\end{prop}
We point out that $\overline D(0) |I_x|^2$ does not appear in the boundary conditions because $\overline D(0)=0$. 
\smallskip

The limit passage in \Cref{prop:phi_convergence} is handled using the half-relaxed limits.  As such, the first and most difficult step is obtaining a priori estimates on $v^\epsilon$ that hold uniformly in $\epsilon$.  A na\"ive approach following~\cite{EvansSouganidis} will fail due to the competing effects of the degeneracy  at $\theta=0$ and the unboundedness at $\theta = \infty$ of $\overline D$.
\smallskip

Recalling  that $u^\epsilon = e^{-v^\epsilon/\epsilon}$, from \Cref{prop:phi_convergence}, one might expect that $u^\e$ converges to one on the zero set of $I$ and zero on the set where $I$ is positive.
This is verified by the following theorem.
\begin{theorem}\label{thm}
	Suppose that \Cref{assumption:D} and \Cref{assumption:u_0} hold, and  let $I$ be the unique solution to~\eqref{eq:I} and~\eqref{eq:I_0_intro}.  Then
	\[
		\lim_{\epsilon\to0} u^\epsilon = \begin{cases}
			0 \quad \text{ uniformly on compact subsets of }  \quad \{I > 0\},\\
			1 \quad \text{ uniformly on compact subsets of }   \quad\Int\{I = 0\}.
			\end{cases}
	\]
\end{theorem}

	\Cref{thm} may be proved  in more generality.  Following the arguments of \cite[Section 4]{EvansSouganidis}, it is clear that we may replace $u(1-u)$ with $f(u)$ for any $f \in C^2$ such that $f(x) > 0$ if $x \in (0,1)$, $f(x) < 0$ if $x \notin[0,1]$, and $f'(0) = \sup_{u\in[0,1]} f(u)/u.$
Then, in~\eqref{eq:I}, each ``$1$'' is replaced by $f'(0)$.  
\smallskip

Unfortunately, $I$ is difficult to compute analytically due to the fact that it is the viscosity solution of a variational inequality. 
In order to characterize the sets $\{I>0\}$ and $\Int\{I=0\}$ more explicitly, we consider the geometric front equation
\begin{equation}\label{eq:w_intro}
\begin{cases}
	w_t + 2\sqrt{\overline D(\theta) |w_x|^2 + |w_\theta|^2} = 0 \quad &\text{ in } \quad  \R\times \R^+\times\R^+,\\
	\max\{- w_\theta, w_t + 2 |w_\theta|\} \geq 0 \quad 
		&\text{ on }  \quad  \R \times\{0\} \times \R^+,\\
	\min\{-w_\theta, w_t + 2|w_\theta|\} \leq 0
		\quad  &\text{ on } \quad  \R \times\{0\} \times \R^+,
\end{cases}
\end{equation}
with 
\begin{equation}\label{eq:w_0_intro}
	w(x,\theta,0) = \begin{cases}
		1 \quad \text{ on } \quad  \overline G_0^c,\\
		0\quad \text{ on } \quad  G_0.
	\end{cases}
\end{equation}
It turns out (see~\Cref{sec:geometric}) that the zero level sets of $w$ and $I$ are comparable.   Indeed, we have:
%
\begin{prop}\label{prop:geometric}
	Suppose that \Cref{assumption:D} and \Cref{assumption:u_0} hold.  Then, there is a unique solution to~\eqref{eq:w_intro} and~\eqref{eq:w_0_intro} and
	\[
		\lim_{\epsilon\to0} u^\epsilon = \begin{cases}
			0 \quad \text{ uniformly on compact subsets of } \quad \Int\{w = 1\},\\
			1\quad \text{ uniformly on compact subsets of }  \quad \Int\{w = 0\}.\\
			\end{cases}
	\]
\end{prop}

It also follows from our  analysis  that we may compare $I$ with the solution to the Hamilton-Jacobi equation coming from the linearized cane toads equation, that is, the equation with $u(1-u)$ replaced by $u$.  The solutions to this equation are more easily computable analytically (see \Cref{sec:front_location}).  \Cref{prop:geometric} is a key tool in establishing this.  Further, it provides a Huygen's principle for the cane toads front; that is, our front moves normal to itself with velocity depending only on the normal vector and its position in $\theta$.
\smallskip

Consider $\mathcal{A}_{x,\theta,G_0,t} = \{\gamma \in H^1((0,t);\R\times \R^+) : \gamma(0) = (x,\theta), \gamma(t) \in \overline G_0\}$ and the action
\begin{equation}\label{eq:action}
\begin{split}
	J(x,\theta,t) := \min_{\mathcal{A}_{x,\theta,G_0,t}} &\int_0^t \left[\frac{\dot\gamma_1(s)^2}{4 \overline D(\gamma_2(s))} + \frac{\dot\gamma_2(s)^2}{4} - 1\right] ds.
\end{split}
\end{equation}
When $\overline D$ is not degenerate, it is well-known that $J$ satisfies a Hamilton-Jacobi equation similar to~\eqref{eq:I} in $\R\times\R^+\times\R^+$, see~\eqref{eq:J}.  In \Cref{sec:appendix}, we show that this can be extended to our setting by showing that the trajectories in~\eqref{eq:action} remain bounded away from $\theta=0$.  This follows from elementary, if fairly complicated, arguments in which we alter any trajectory that approaches the boundary $\R\times\{0\}$ to obtain a new trajectory that is ``more optimal.''  Also, due to the degeneracy in $\overline D$, $J$ satisfies Neumann boundary conditions (see \Cref{sec:J}).
\begin{prop}\label{prop:I_J}
	Suppose that  \Cref{assumption:D} and \Cref{assumption:u_0} hold.  Then 
	\[
		\lim_{\epsilon\to0} u^\epsilon = \begin{cases}
			0 \quad \text{ uniformly on compact subsets of }  \quad \{J > 0\},\\
			1  \quad \text{ uniformly on compact subsets of } \quad  \{J < 0\}.
			\end{cases}
	\]
\end{prop}

We prove \Cref{prop:geometric} and \Cref{prop:I_J} simultaneously by showing that $\Int\{w = 1\} = \{J>0\} = \{I > 0\}$ and $\Int\{w = 0\} = \{J<0\} = \Int\{I = 0\}$ and applying \Cref{thm}.  Some of these inclusion follow by the maximum principle applied as in~\cite{MajdaSouganidis}.  To obtain the last inclusion $\Int\{w = 1\} \subset \{J > 0\}$, we use a characterization of $w$ in terms of trajectories that is analogous to~\eqref{eq:action}.  This last step differs from~\cite{MajdaSouganidis}, where the homogeneity of the equation considered there admits a more direct argument.
\smallskip

We point out that \Cref{prop:I_J} shows that the solutions are pulled.  In other words, the propagation speed depends only on the linearized equation at the highest order.  We also remark that there are examples where this is not the case; see \cite{MajdaSouganidis} for a discussion of this phenomena and for the construction of some counter-examples. 
\smallskip

\begin{figure}
\begin{center}
\begin{overpic}[scale = .35]
		{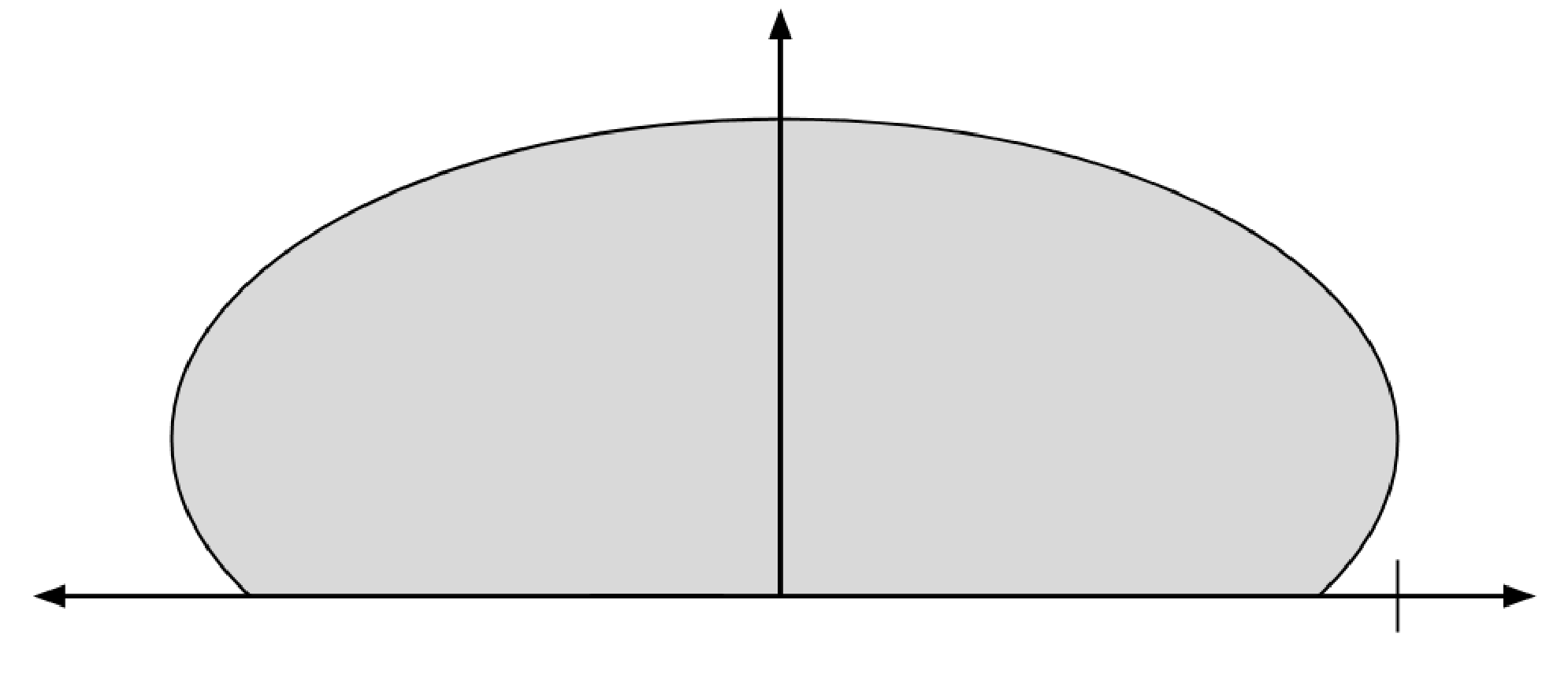}
	\put (79,-.5){Front location}
	\put (100,4){$x$}
	\put (49,43){$\theta$}	
	\put (58, 18){$u \sim 1$}
	\put (84,33) {$u \sim 0$}
\end{overpic}
\caption{A zoomed out cartoon of a typical solution to~\eqref{eq:cane_toads_unscaled} with $t\gg 1$.  The gray region is where $u \sim 1$, the white region is where $u \sim 0$ and the black boundary is where $u$ transitions from $1$ to $0$.  The ``front location'' is the furthest right point in $x$ where $u$ transitions between $1$ and $0$.}\label{fig:front_location}
\end{center}
\end{figure}
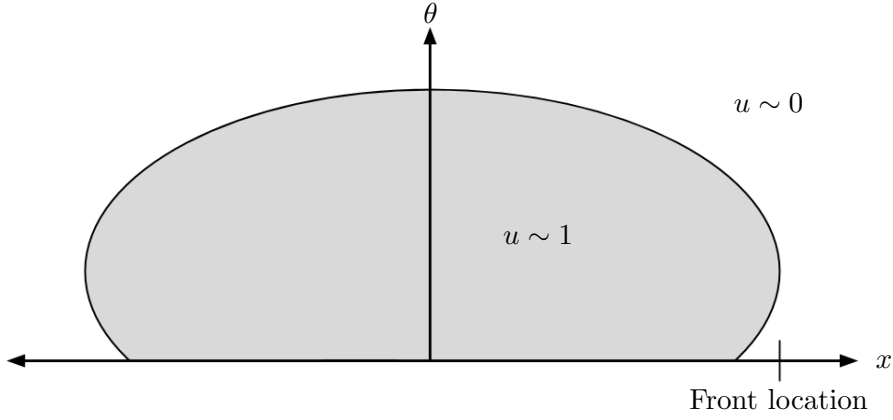

In fact, due to the connection between $u$, $I$, and $J$ given by \Cref{thm} and \Cref{prop:I_J}, the front location in physical space (see \Cref{fig:front_location}) for the problem~\eqref{eq:cane_toads_unscaled} when $\e=1$ can be obtained through the level set $\{J=0\}$.  Since $\overline D$ has the simplified form $\overline D(\theta) = \theta^p$ for some $p>0$, this level set can be explicitly computed analytically.  We discuss, in \Cref{sec:front_location}, how to compute and prove these asymptotics. 
Indeed, in \Cref{sec:front_location}, we see that, for $t\gg 1$,
\begin{equation}\label{eq:front_location_formula}
	\text{Front location in $x$ at time $t$}
		\sim \Bigg(\frac{8}{2 + p} \frac{\Gamma\big(\frac{1}{2} + \frac{1}{p}\big)^{p/2}}{2^{1 - p/2} \pi^{p/4} \Gamma\big(1 + \frac{1}{p}\big)^{p/2}}\Bigg) t \sqrt{D(t)}.
\end{equation}
In particular we see that, while the order of the front location is determined by $D$, the coefficient in front depends only on the limiting problem~\eqref{eq:I}.
\smallskip

Returning to the example~\eqref{eq:D_example} and the solution to~\eqref{eq:cane_toads_unscaled} with this choice of $D$, we see that $\overline D(\theta) = \theta$.   Hence, $p=1$ in~\eqref{eq:front_location_formula} and we recover the front location $\sim (4/3) t^{3/2} \sqrt{\log(t)}$.


\smallskip


The approach that we follow here is based on the work of Freidlin~\cite{Freidlin1,Freidlin2}, \cite{EvansSouganidis}, Barles, Evans, and Souganidis~\cite{BarlesEvansSouganidis}, and \cite{MajdaSouganidis}.  In the cane toads equation introduced by Benichou et.~al.~\cite{BenichouEtAl} $u(1-u)$ is replaced by the non-local term $u(1-\int u d\theta)$ and $D(\theta) = \theta$.  In that setting and with the additional assumption that the trait $\theta$ takes values between two fixed positive constants $[\underline\theta,\overline\theta]$, Bouin and Calvez \cite{BouinCalvez} proved the existence of traveling waves, Turanova \cite{Turanova} showed  that the speed of the traveling wave governs the spread of the population in the Cauchy problem, and Bouin, Henderson, and Ryzhik~\cite{BHR_delay} established  a Bramson-type  logarithmic delay between the speed of the slowest traveling wave and the location of the front for any initially ``localized'' solution to the Cauchy problem.  When the trait space is unbounded, as in this work, Bouin et.~al.~\cite{BCMetal} predicted that the location of the front is of order $(4/3)t^{3/2}$.  This was then verified in the local model by Berestycki, Mouhot, and Raoul \cite{BerestyckiMouhotRaoul} and by Bouin, Henderson, and Ryzhik \cite{BHR_acceleration}  using probabilistic and analytic techniques, respectively.  It was also shown in  \cite{BerestyckiMouhotRaoul}  that in a windowed non-local model the propagation speed is the same, while \cite{BHR_acceleration} obtained weak bounds of order $t^{3/2}$ for the full non-local model.  A model with a trade-off term, that is, a penalization for large of traits, has been proposed and studied by Bouin, Chan, Henderson, and Kim~\cite{BCHK}.  In the present article, we investigate only the local model as the non-local model  exhibits quite different behavior, see \cite{CalvezHendersonMirrahimiTuranova}. 
We also mention related works on finite domains by Perthame and Souganidis~\cite{PerthameSouganidis} and Lam and Lou~\cite{LamLou}.

\subsubsection*{Outline of the paper}  We begin by proving \Cref{thm} in \Cref{sec:thm} assuming \Cref{prop:phi_convergence}.  In \Cref{sec:limits}, we prove \Cref{prop:phi_convergence} using the half-relaxed limits along with uniqueness of the limiting Hamilton-Jacobi equations.  New ingredients in this step are the a priori estimates, which are more difficult to obtain since the Hamiltonian is degenerate at $\theta=0$ and unbounded at $\theta = +\infty$, and the boundary conditions, since boundaries did not appear in earlier thin-front limit works.  In \Cref{sec:relationship}, we prove \Cref{prop:geometric} and \Cref{prop:I_J}, that is, the propagation of $u$ is characterized by the solution to the geometric front equation, $w$, and the solution to the linearized problem, $J$.  Again, the boundary conditions provide the main difficulties in this section.  We include brief comments in \Cref{sec:appendix} describing why we may import the representation formulas for $w$ and $J$ from the boundary-less setting.  We conclude the paper with a discussion and computation of the front location in \Cref{sec:front_location}.

\subsubsection*{Notion of solution}
Throughout this work, we employ the concept of viscosity solutions, along with the vocabulary accompanying it.  The interested reader should consult one of the many references, for example, the ``User's Guide'' of Crandall, Ishii, and Lions~\cite{CrandallIshiiLions}.

\subsubsection*{Acknowledgments}
CH was partially supported by the National Science Foundation Research Training Group grant DMS-1246999.
BP has been supported by the French ``ANR blanche" project Kibord:  ANR-13-BS01-0004.
PS was partially supported by the National Science Foundation grants DMS-1266383 and DMS-1600129 and the Office for Naval Research Grant N00014-17-1-2095.

\section{The proof of \Cref{thm}}\label{sec:thm}


The proof hinges on the locally uniform convergence of $v^\epsilon$ to $I$ guaranteed by \Cref{prop:phi_convergence}.  We show how to conclude \Cref{thm} assuming this proposition, which is proved in \Cref{sec:limits}.   Our proof follows the general outline of~\cite{EvansSouganidis}, with the relevant modifications made to handle the technical issues arising from the boundary.

\begin{proof}[Proof of \Cref{thm}]
	 \sloppy We first consider the set $\{I > 0\}$.  Fix any point $(x_0,\theta_0,t_0)$ such that $I(x_0,\theta_0,t_0) > 0$ with $t_0 > 0$.  
Since $v^\epsilon$ converges to $I$ locally uniformly as $\e$ tends to zero by~\Cref{prop:phi_convergence},  $v^\epsilon(x,\theta,t) > \delta$ for some $\delta,r>0$ and any $(x,\theta,t) \in B_r(x_0,\theta_0,t_0)$ when $\e$ is sufficiently small.  It follows that  $u^\epsilon(x,\theta,t) \leq \exp\left\{ -\delta/\epsilon\right\}$ for all $\epsilon$ sufficiently small and all $(x,\theta,t) \in B_r(x_0,\theta_0,t_0)$.  Hence $u^\epsilon$ converges to zero uniformly on $B_r(x_0,\theta_0,t_0)$ as $\epsilon$ tends to zero.
\smallskip

	Now we consider the set $\Int\{I = 0\}$.  Fix $(x_0,\theta_0,t_0)\in \Int\{I = 0\}$.  There are two cases to investigate depending on whether $\theta_0$ is positive or zero.  First assume that $\theta_0>0$. 
Define a test function $\psi(x,\theta,t) = |t-t_0|^2 + |x-x_0|^2 + |\theta-\theta_0|^2,$
and note that, since $I \equiv 0$ near $(x_0,\theta_0,t_0)$,  $I - \psi$ has a strict local maximum at $(x_0,\theta_0,t_0)$ on a small enough ball centered at $(x_0,\theta_0,t_0)$.   It follows that $v^\epsilon - \psi$ has a maximum at some point $(x_\epsilon, \theta_\epsilon, t_\epsilon)$ such that $(x_\epsilon,\theta_\epsilon,t_\epsilon)$ converges to $(x_0,\theta_0,t_0)$ as $\epsilon$ tends to zero.  
\smallskip	
	
	Because $\theta_0 > 0$, we may restrict to $\epsilon$ sufficiently small so that $\theta_\epsilon > 0$.  Then, using~\eqref{eq:phase}, we find, at $(x_\e,\theta_\e,t_\e)$,
$\psi_t - \epsilon \overline D^\epsilon  \psi_{xx} - \epsilon \psi_{\theta\theta} + \overline D^\epsilon |\psi_x|^2 + |\psi_\theta|^2 \leq  u^\epsilon - 1.$
	An explicit computation, using only the form of $\psi$ and the fact that $(x_\e,\theta_\e,t_\e)$ converges to $(x_0,\theta_0,t_0)$ as $\e$ tends to zero, shows that the left hand side tends to zero as $\epsilon$ tends to zero.  We infer that $1 \leq \liminf_{\e\to0} u^\epsilon(x_\epsilon,\theta_\epsilon,t_\epsilon)$.  On the other hand, recall that $( x_\epsilon,\theta_\epsilon, t_\epsilon)$ is the location of a minimum of $u^\epsilon \exp\{\psi/\epsilon\}$.  Hence we have that
	\[
		\liminf_{\epsilon\to0}u^\epsilon(x_0,\theta_0,t_0)
			\geq \liminf_{\epsilon\to0} u^\epsilon(x_\epsilon,\theta_\epsilon,t_\epsilon)\exp\left\{\e^{-1}\left(|t_\epsilon - t_0|^2 + |x_\epsilon-x_0|^2 + |\theta_\epsilon-\theta_0|^2\right)\right\}
			\geq 1.
	\]
	Initially, $u_0 \leq 1$ in $\R\times [0,\infty)$. The maximum principle implies that $u^\epsilon \leq 1$ in $\R\times[0,\infty)\times[0,\infty)$ for all $\epsilon$. It follows that $1 \geq \limsup_{\e\to0}u^\e(x_0,\theta_0,t_0)$.  As a consequence, $\limsup_{\e\to0}u^\e(x_0,\theta_0,t_0) = \liminf_{\e\to0} u^\e(x_0,\theta_0,t_0) = 1$, which implies that $u^\e(x_0,\theta_0,t_0)$ converges to $1$ as $\e$ tends to zero.  This concludes the proof in the case that $\theta_0 = 0$.
\smallskip
	
	If $\theta_0 = 0$,  define
$\psi^\epsilon(x,\theta,t) := |t-t_0|^2 + |x-x_0|^2 + |\theta-\epsilon^2|^2,$
	and let $(x_\epsilon,\theta_\epsilon,t_\epsilon)$ be a maximum of $v^\epsilon - \psi^\epsilon$.  Since $\psi^\epsilon$ and $v^\e$ converge to $\psi$ and $I$, respectively, as $\e$ tends to zero and $I - \psi$ has a strict local maximum at $(x_0,\theta_0,t_0)$, it follows that  $(x_\epsilon, \theta_\epsilon,t_\epsilon)$ converges to $(x_0, 0,t_0)$ as $\e$ tends to zero.  
	\smallskip

	We claim that $\theta_\e > 0$ for all $\e>0$, and we proceed by contradiction.  Suppose that $\theta_\epsilon = 0$ for any $\epsilon>0$.  Because $v^\epsilon - \psi$ has a local maximum at $(x_\epsilon, 0, t_\epsilon)$,     
$v_\theta^\epsilon(x_\epsilon, 0,t_\epsilon) \leq \psi_\theta^\epsilon(x_\epsilon, 0,t_\epsilon).$
By~\eqref{eq:phase},  the left hand side is $0$. The right hand side  is, by construction,  $-2\epsilon^2$.  This is a contradiction.
	\smallskip
	
	It follows that $\theta_\epsilon > 0$ for all $\e>0$. 
	Then \eqref{eq:phase} yields, at $(x_\epsilon,\theta_\epsilon,t_\epsilon)$,
	\[
		\psi_t^\epsilon - \epsilon \overline D^\epsilon(\theta) \psi_{xx}^\epsilon - \epsilon\psi_{\theta\theta}^\epsilon + \overline D^\epsilon(\theta)|\psi_x^\epsilon|^2 + |\psi_\theta^\epsilon|^2 \leq u^\epsilon -1.
	\]
	As above, an explicit computation shows that the left hand side tends to zero as $\e$ tends to zero.  Hence, $\liminf u^\epsilon(x_\epsilon,\theta_\epsilon,t_\epsilon) \geq 1$. 
	\smallskip
	
	By construction, $(x_\epsilon,\theta_\epsilon,t_\epsilon)$ is the location of a local minimum of $u^\e\exp\{\psi^\e/\e\}$.  Thus,
	\[
		\liminf u^\epsilon(x_0,0,t_0)
			\geq \liminf u^\epsilon(x_\epsilon,\theta_\epsilon,t_\epsilon) \exp\left\{ \frac{|t_\epsilon - t_0|^2 + |x_\epsilon-x_0|^2 + |\theta_\epsilon - \epsilon^2|^2}{\epsilon} - \epsilon \right\}
				\geq 1.
	\]
From the conclusion of the previous case, when $\theta_0>0$, recall that $u^\e \leq 1$, which immediately yields $\limsup_{\e\to0} u^\e(x_0,\theta_0,t_0)\leq 1$.  Arguing as above, we conclude that $u^\e(x_0,0,t_0)$ converges to $1$ as $\e$ tends to zero.  This concludes the proof.
\end{proof}

\section{The limit of the sequence $(v^\epsilon)_{\e>0}$ -- the proof of \Cref{prop:phi_convergence}}\label{sec:limits}

We proceed in three steps.  In the first, we obtain uniform bounds on $v^\epsilon$ on compact subsets of $(G_0 \times \{t=0\}) \cup ( \R\times [0,\infty) \times \R^+)$.  In the second, we take the half-relaxed limits of the sequence $(v^\epsilon)_{\e>0}$ to obtain $v_*$ and $v^*$, and we show that they are respectively super- and sub-solutions of~\eqref{eq:I}.  Finally in the last step, we use comparison to show that $v_* = v^* = I$ and conclude that $v^\epsilon$ converges locally uniformly to $I$.

\subsection{An upper bound for $v^\epsilon$}\label{sec:apriori}

By the maximum principle, $0 \leq u^\epsilon \leq 1$ and so $v^\epsilon \geq 0$.  In order to take the half-relaxed limits, we need an upper bound on $v^\epsilon$ that is uniform in $\e$.

\begin{lem}\label{lem:apriori}
	Suppose that \Cref{assumption:D} and \Cref{assumption:u_0} hold.  Fix any compact subset  $Q$ of 
$
	\left(G_0 \times \{t = 0\}\right) \cup \left(\R\times[0,\infty)\times\R^+\right).
$
There exists $C=C(Q)>0$ such that, if $(x,\theta,t) \in Q$, then $v^\epsilon(x,\theta,t)
		\leq C.$
Further, if $Q \subset G_0 \times [0,\infty)$, then there exists a constant $C' = C'(Q)$ such that, if $(x,\theta,t)\in Q$, then
\begin{equation}\label{eq:small_on_G_0}
	v^\e(x,\theta,t) \leq \e \ C'.
\end{equation}
\end{lem}
\begin{proof}
We begin by noticing that, when $\epsilon > 0$,   we may ignore the boundary $\{\theta = 0\}.$  Indeed, using the Neumann boundary condition, we may extend $u^\epsilon$, and thus $v^\epsilon$, evenly to $\R\times\R\times \R^+$.  The parabolic regularity theory yields that $v^\epsilon$ satisfies~\eqref{eq:phase} on $\R\times\R\times\R^+$ with  $\overline D^\epsilon(\theta)$ replaced by $\overline D^\epsilon(|\theta|)$;  for more details see~\cite{Turanova}.  For the remainder of this proof, we abuse notation by letting $u^\epsilon$ and $v^\epsilon$ refer to their even extensions.
\smallskip

Next, we set some notation.  For any $R>0$ and $(x_0,\theta_0) \in \R\times \R$, let
\[
	Q_R(x_0,\theta_0) :=
		(x_0 - R, x_0 + R)\times (\theta_0 - R, \theta_0 + R).
\]
In the sequel, we use $Q_R$ to refer to $Q_R(0,0)$.
\smallskip

We proceed in two steps.  First, for any $T,R>0$ and $(x_0,\theta_0)$ such that $|\theta_0|>R/2$ and $\overline Q_R(x_0,\theta_0) \subset G_0$, we build a barrier on $Q_R(x_0,\theta_0)\times[0,T]$ that yields an upper bound on $v^\e$ in $Q_{3R/4}(x_0,\theta_0) \times [0,T]$ that is uniform in $\e$.  Since $G_0$ is open, it is easy to see that
\[
	G_0 = \bigcup_{R\in (0,1)} \bigcup_{\substack{|\theta_0| > R/2, \\ \overline{Q_R}(x_0,\theta_0) \subset G_0}} Q_{3R/4}(x_0,\theta_0).
\]
Thus, the bound we have is enough to conclude an upper bound on $v^\e$ that is independent of $\e$ on any compact subset of $G_0 \times [0,T]$.  The second step extends this by building a barrier on sets of the form $(Q_L(x_0,\theta_0)\setminus \overline{Q_{R/2}}(x_0,\theta_0)) \times [T^{-1},T^{-1} + T]$, where  $R$, $x_0$, and $\theta_0$ are as in the first step, $T>1$, and $L>R$. This crucially uses the bound obtained in the first step to control the portion of the parabolic boundary $\partial Q_{R/2}\times(T^{-1},T^{-1}+T)$.  This provides an upper bound on $v^\e$ that is independent of $\e$ on compact subsets of $\R \times \R \times \R^+$, which finishes the proof.
\smallskip

Our proof follows the ideas of~\cite{EvansSouganidis} with a few key modifications, which we mention as they arise.  The added complication that occurs in our proof is due to the interplay of the degeneracy of $\overline D^\e$ at $\theta=0$ and its growth at $\theta = \infty$.  We point out that a crucial observation that saves our computations is restricting to cubes $Q_R(x_0,\theta_0)$ where $|\theta_0|>R/2$, see Step Two below.
\smallskip

{\bf \# Step one:} 
Since the equation is translation invariant in $x$, we may assume that $x_0 = 0$, without loss of generality.  We may also assume, without loss of generality, that $\theta_0 > 0$, which, in turn, implies that $\theta_0 > R/2$.  For notational ease, we translate the equation in $\theta$.  That is, we define $\nu(x,\theta,t) = v^\e(x, \theta+\theta_0, t)$ and $\overline D^\e_0(\theta) = \overline D^\e(\theta+\theta_0)$.  It follows  that,
\begin{equation}\label{eq:shifted_phase}
	\nu_t - \e \overline D^\e_0 \nu_{xx} - \e \nu_{\theta\theta} + \overline D^\e |\nu_x|^2 + |\overline \nu_\theta|^2 + 1 - e^{-\nu/\e} = 0
		\qquad \text{  in $\R \times \R\times \R^+$}.
\end{equation}
An upper bound on $\nu$ in $Q_R$ implies the desired bound on $v^\e$ in $Q_R(0,\theta_0)$.
\smallskip

We proceed by building a barrier.  Consider, for $\alpha$, $\beta$, and $\rho$ that are positive constants to be determined, the auxiliary function
\[
	\psi(x,\theta,t) := \alpha t + \beta + \frac{\rho}{R^2 - x^2} + \frac{\rho}{R^2 - \theta^2}
		\qquad \text{ in } Q_R\times \R^+.
\]
We point out that $\psi$ differs from the barrier used in~\cite{EvansSouganidis}, and this difference simplifies many computations because it separates the variables.
\smallskip

Straightforward calculations yield
\begin{equation}\label{eq:super-solution_a}
\begin{split}
	&\psi_t - \e \overline D^\e_0 \psi_{xx} - \e \psi_{\theta\theta} + \overline D^\e_0 |\psi_x|^2 + |\psi_\theta|^2 + 1 - e^{-\psi/\e}\\
		&\geq \alpha  - \e \rho \left( \overline D^\e_0 \left(\frac{2}{(R^2- x^2)^2} + \frac{8 x^2}{(R^2-x^2)^3}\right) + \left(\frac{2}{(R^2- \theta^2)^2} + \frac{8 \theta^2}{(R^2-\theta^2)^3} \right) \right)\\
		&\qquad	+ 4\rho^2 \left(\overline D^\e_0 \frac{x^2}{(R^2 - x^2)^4} + \frac{\theta^2}{(R^2 - \theta^2)^4}\right)\\
		&= \alpha  + \frac{2\rho \overline D^\e_0}{(R^2 - x^2)^3} \left( \frac{2 \rho x^2}{R^2-x^2} - \e(R^2 + 3x^2)\right)
			+ \frac{2\rho}{(R^2 - \theta^2)^3} \left(\frac{2 \rho \theta^2}{R^2 - \theta^2} - \e(R^2 + 3\theta^2) \right).
\end{split}
\end{equation}
We define
\begin{equation}\label{eq:alpha_rho}
\begin{split}
	&\rho := 6\e R^2,~~
	\beta := \displaystyle\max_{Q_R} \nu(x,\theta,0)
		= \displaystyle\max_{Q_R(0,\theta_0)} v^\e_0(x,\theta), ~~
	\text{and}~~
	\alpha := \frac{ 20 \e \rho \displaystyle\max_{|\theta'|\leq R}( 1+\overline D^\e_0(\theta'))}{R^4}.
\end{split}
\end{equation}
\smallskip

Consider the second term in the last line of~\eqref{eq:super-solution_a}.  When $|x| \in [R/2,R]$, we have
\[\begin{split}
	\frac{2 \rho x^2}{R^2-x^2} - &\e\left(R^2 + 3x^2\right)
		\geq \frac{2 \rho (R/2)^2}{R^2 - (R/2)^2} - \e \left(R^2 + 3R^2\right)
		= \frac{2\rho}{3} - 4 \e R^2
		= 0,
\end{split}\]
where the last equality follows from the definition of $\rho$~\eqref{eq:alpha_rho}.  When $|x| \in [0,R/2]$, we have
\[\begin{split}
	\frac{2\rho \overline D^\e_0}{(R^2 - x^2)^3} \left( \frac{2 \rho x^2}{R^2-x^2} - \e(R^2 + 3x^2)\right)
		&\geq  \frac{2\rho \displaystyle \max_{|\theta'| \leq R} \overline D^\e_0}{(R^2 - (R/2)^2)^3} ( 0 - 2\e R^2)
		\geq - \frac{\alpha}{2}.
\end{split}\]
We see that, for all $(x,\theta) \in Q_R$,
\[
	\frac\alpha2  + \frac{2\rho \overline D^\e_0}{(R^2 - x^2)^3} \left( \frac{2 \rho x^2}{R^2-x^2} - \e(R^2 + 3x^2)\right)
		\geq 0.
\]
A similar argument shows that, for all $(x,\theta) \in Q_R$,
\[
	\frac\alpha2 + \frac{2\rho}{(R^2 - \theta^2)^3} \left(\frac{2 \rho \theta^2}{R^2 - \theta^2} - \e(R^2 + 3\theta^2) \right)
		\geq 0.
\]
These two inequalities, applied to~\eqref{eq:super-solution_a}, show that
\[
	\psi_t - \e \overline D^\e_0 \psi_{xx}  - \psi_{\theta\theta} + \overline D^\e_0 |\psi_x|^2 + |\psi_{\theta}|^2 + 1 - e^{-\psi/\e} 
		\geq 0
	\qquad \text{ in } Q_R\times \R^+;
\]
that is, $\psi$ is a super-solution of~\eqref{eq:shifted_phase} in $Q_R\times \R^+$.
\smallskip

Next, the choice of $\beta$ ensures that, on $Q_R$, ${\psi(\cdot,\cdot,0)} \geq \beta \geq v^\epsilon_0.$
Further, the strong maximum principle implies that $u^\e > 0$ on $\R\times \R\times \R^+$, which implies that $v^\e$, and thus $\nu$, is finite in $\R\times\R\times \R^+$.  Because $\psi = + \infty$ on $\partial Q_R\times \R^+$,  $\psi \geq \nu$ on $\partial Q_R\times \R^+$.  The maximum principle implies that $0 \leq \nu \leq \psi$ on $Q_R \times\R^+$.  In particular, there exists some  $C_R>0$, which depends only on $\theta_0$, $R$, $D$, and $u_0$,  such that, on $Q_{3R/4} \times [0,\infty)$,
\begin{equation}\label{eq:preliminary_bound}
	v^\epsilon \leq \psi \leq C_R(1+t). 
\end{equation}

We now establish~\eqref{eq:small_on_G_0}. Since $\overline{Q_R}(0,\theta_0) \subset G_0 = \{u_0 > 0\}$, then
 \[
 	\beta = \max_{(x,\theta) \in \overline{Q_R}(0,\theta_0)} v^\epsilon_0 \leq \epsilon \log\left(\frac{1}{\min_{(x,\theta) \in \overline{Q_R}(0,\theta_0)} u_0(x,\theta)}\right).
\]
We recall that $\min_{(x,\theta) \in \overline {Q_R}(0,\theta_0)}
 u_0 > 0$ due to the continuity of $u_0$. 
Also, it follows from their definitions that $\alpha, \rho \leq C \epsilon$, for some constant $C$ depending only on $D$, $R$, and $\theta_0$.  We conclude that, for any $(x_0,\theta_0)$ and $R,T>0$ such that $\overline{Q_R}(x_0,\theta_0) \subset G_0$ and $|\theta_0| > R/2$, there exists a constant $C$ that depends only on $u_0$, $R$, $D$, $(x_0,\theta_0)$, and $T$ such that
$\nu \leq C\e$ in $Q_{3R/4}(x_0,\theta_0) \times [0,T].$
Given a compact subset $Q \subset G_0 \times[0,\infty)$, it can be covered by finitely many sets of the form $Q_{3R/4}(x_0,\theta_0) \times [0,T]$ where $\overline{Q_R}(x_0,\theta_0) \subset G_0$ and $|\theta_0|>R/2$.  Hence, we conclude that, for any such $Q$, there exists $C = C(Q)$ such that $\nu \leq C\e$ on $Q$;
that is,~\eqref{eq:small_on_G_0} holds.
\smallskip

{\bf \# Step two:} Let $R$ and $\theta_0$ be as above and fix $L> R$ and $T>1$.
Define $\nu(x,\theta, t) = v^\e(x,\theta+\theta_0, t + T^{-1})$.  Then, $\nu$ satisfies~\eqref{eq:shifted_phase}.  In view of the bound~\eqref{eq:preliminary_bound}, a bound on $\nu$ in $Q_L\setminus \overline Q_{R/2} \times [0,T]$, yields a bound on $v^\e$ on $Q_L(0,\theta_0)\times [T^{-1},T^{-1} + T]$.  To obtain such a bound, we build a barrier.  Before beginning, we note that, in~\cite{EvansSouganidis}, the authors are able to construct a barrier on their analogue of $\overline{Q_{R/2}}^c \times \R^+$ directly.  This approach will not work in our setting since $\overline D^\e_0$ is unbounded.  This is the reason that we strict to cubic annuli in physical and trait space.
\smallskip

Define
\begin{equation}\label{eq:beta_rho}
\begin{split}
	&\beta := \max_{\partial Q_{R/2} \times [0,T]} \nu
		= \max_{\partial Q_{R/2}(0,\theta_0) \times \left[\frac1T, \frac1T+T\right]} v^\e
	~~\text{and}~~
	\rho := \frac{2}{R^2}\frac{4 \e T L^4 (1 + \displaystyle\max_{|\theta'|\leq L} \overline D^\e_0(\theta')) + L^8}{\displaystyle\min\left\{1,\min_{|\theta'|\leq R/2}\overline D^\e_0(\theta')\right\}}.
\end{split}
\end{equation}
Since $\theta_0 > R/2$, it follows that the denominator of $\rho$ is bounded below by a positive constant independent of $\e$.  Also, $\beta$ is bounded above depending only on $u_0$, $\theta_0$, $R$, and $T$, due to~\eqref{eq:preliminary_bound}. Then define
\[
	\zeta\left(x,\theta, t + \frac{1}{T}\right)
		:= \beta + \frac{\rho}{t(L^2 - x^2)} + \frac{\rho}{t(L^2 - \theta^2)}
			\quad \text{ in } Q_L \setminus \overline Q_{R/2} \times [0,T].
\]
Note that the restriction $\theta_0 > R/2$ has the consequence that when $\overline D^\e_0 \sim 0$, $|\zeta_\theta| \sim O(1)$.  This is the key observation in constructing $\zeta$ that allows us to side-step any complications stemming from the degeneracy $\overline D^\e_0(-\theta_0) = \overline D^\e(0) = 0$.
\smallskip

Also, note that this barrier is different from the one constructed in~\cite{EvansSouganidis}.  Indeed, since we are restricted to a compact set in physical and trait space, it is crucial that $\zeta$ be larger than $\nu$ on the boundary $\partial Q_L$.  Hence, we may not use the quadratically growing barrier from~\cite{EvansSouganidis}.
\smallskip

We show that $\zeta$ is super-solution of~\eqref{eq:shifted_phase}.  A straightforward computation yields
\begin{equation}\label{eq:c319}
\begin{split}
	\zeta_t - &\e\overline D_0^\e \zeta_{xx} - \e\zeta_{\theta\theta} + \overline D^\e_0 |\zeta_x|^2 + |\zeta_\theta|^2 + 1 - e^{-\zeta/\e}\\
		&\geq   - \frac{\rho}{t^2(L^2 - x^2)} - \frac{\rho}{t^2(L^2 - \theta^2)}
			- \e\rho\overline D_0^\e \left( \frac{2}{t(L^2 - x^2)^2} + \frac{8 x^2}{t(L^2 - x^2)^3} \right)\\
		&\qquad - \e\rho \left( \frac{2}{t(L^4 - \theta^4)^2} + \frac{8 \theta^2}{t(L^2 - x^2)^3} \right)
			+ \overline D_0^\e \frac{4 \rho^2 x^2}{t^2(L^2 - x^2)^4}
			+ \frac{4 \rho^2 \theta^2}{t^2(L^2 - \theta^2)^4}\\
		&=  \frac{2\rho}{t^2}\Bigg[ \frac{1}{(L^2 - x^2)^4}\left(2 \rho x^2  \overline D^\e_0  - \e  \overline D^\e_0 t(L^2 + 3 x^2)(L^2-x^2) - \frac{(L^2-x^2)^4}{2} \right)\\
			&\qquad  + \frac{1}{(L^2 - \theta^2)^4}\left(2 \rho \theta^2 - \e t(L^2 + 3 \theta^2)(L^2-x^2) - \frac{(L^2-\theta^2)^4}{2}\right) \Bigg].
\end{split}
\end{equation}
Since $(x,\theta) \in Q_L\setminus \overline Q_{R/2}$, we consider three cases: (1) $|x| > R/2 \geq |\theta|$; (2) $|\theta| >R/2 \geq |x|$; and (3) $|x|, |\theta| > R/2$.
\smallskip

{\bf Case one:} If $|x| > R/2 \geq |\theta|$, notice that
\begin{equation}\label{eq:c319a}
\begin{split}
	&\frac{1}{(L^2 - x^2)^4}\left(2 \rho x^2  \overline D^\e_0  - \e \overline D^\e_0 t(L^2 + 3 x^2)(L^2-x^2) - \frac{(L^2-x^2)^4}{2}\right) \\
		&\quad\geq \frac{1}{(L^2 - x^2)^4}\left( \frac{\rho R^2}{2}  \min_{|\theta'|\leq R/2}\overline D^\e_0(\theta')  - 4 \e T L^4 \max_{|\theta'| \leq L}\overline D^\e_0(\theta') - \frac{L^8}{2}\right)\\
		&\quad \geq \frac{1}{(L^2 - x^2)^4} \left( 4 \e T L^4 + \frac{L^8}{2}\right)
		\geq \frac{1}{(L^2 - (R/2)^2)^4} \left( 4 \e T L^4 + \frac{L^8}{2}\right),
\end{split}
\end{equation}
where we used the definition of $\rho$~\eqref{eq:beta_rho} in the second-to-last inequality.  On the other hand,
\[\begin{split}
	\frac{1}{(L^2 - \theta^2)^4}&\left(2 \rho \theta^2  - \e t(L^2 + 3 \theta^2)(L^2-\theta^2) - \frac{(L^2-\theta^2)^4}{2} \right)
	\geq \frac{-1}{(L^2 - (R/2)^2)^4} \left(4\e TL^4 + \frac{L^8}{2}\right).
\end{split}\]
Summing these two inequalities and recalling~\eqref{eq:c319} yields
\[
	\zeta_t - \e\overline D_0^\e \zeta_{xx} - \e\zeta_{\theta\theta} + \overline D^\e_0 |\zeta_x|^2 + |\zeta_\theta|^2 + 1 - e^{-\zeta/\e}
		\geq 0
\]
when $|x| >R/2$ and $|\theta| \leq R/2$.
\smallskip

{\bf Case two:} If $|\theta|>R/2 \geq |x|$, the argument is handled in exactly the same way, so we omit it and conclude again that
\[
	\zeta_t - \e\overline D_0^\e \zeta_{xx} - \e\zeta_{\theta\theta} + \overline D^\e_0 |\zeta_x|^2 + |\zeta_\theta|^2 + 1 - e^{-\zeta/\e}
		\geq 0
\]
when $|x| \leq R/2$ and $|\theta| > R/2$.
\smallskip

{\bf Case three:} if $|x|, |\theta| > R/2$, then, following the argument in~\eqref{eq:c319a} in case one, we see that $2\rho x^2 - \e t(L^2 + 3x^2)(L^2 - x^2) \geq 0$.  Hence,
\[\begin{split}
	&\frac{1}{(L^2 - x^2)^4}\left(2 \rho x^2  \overline D^\e_0  - \e \overline D^\e_0 t(L^2 + 3 x^2)(L^2-x^2) - \frac{(L^2-x^2)^4}{2} \right)
		\geq  - \frac12.
\end{split}\]
Also, arguing similarly as in~\eqref{eq:c319a} and using the definition of $\rho$~\eqref{eq:beta_rho}, we find
\[
\begin{split}
	&\frac{1}{(L^2 - \theta^2)^4}\left(2 \rho \theta^2  - \e t(L^2 + 3 \theta^2)(L^2-\theta^2) - \frac{(L^2-\theta^2)^4}{2}\right) \\
		&\quad\geq \frac{1}{(L^2 - \theta^2)^4}\left( \frac{\rho R^2}{2}  - 4 \e T L^4  - \frac{L^8}{2}\right)
		\geq \frac{1}{(L^2 -\theta^2)^4} \left( 4 \e T L^4 + \frac{L^8}{2}\right)
		\geq \frac{1}{2}.
\end{split}
\]
Summing these two inequalities and recalling~\eqref{eq:c319} implies that
\[
	\zeta_t - \e\overline D_0^\e \zeta_{xx} - \e\zeta_{\theta\theta} + \overline D^\e_0 |\zeta_x|^2 + |\zeta_\theta|^2 + 1 - e^{-\zeta/\e}
		\geq 0
\]
when $|x|, |\theta| >R/2$.
\smallskip

The combination of all three cases above implies that $\zeta$ is a super-solution of~\eqref{eq:shifted_phase} in $(Q_L \setminus \overline Q_{R/2})\times (0, T)$.  By the definition of $\beta$~\eqref{eq:beta_rho}, it follows that $\zeta \geq \nu$ on $\partial Q_{R/2} \times [0,T]$.  Also, since $\nu$ is finite on $\overline{Q_L}\times [0,T]$ (see the discussion at the end of Step One) and $\zeta = +\infty$ on $(Q_L\setminus \overline Q_{R/2})\times \{t=0\}$ and on $\partial Q_L \times [0,T]$, then $\zeta \geq \nu$ on $(Q_L\setminus \overline Q_{R/2})\times \{t=0\}$ and on $\partial Q_L \times [0,T]$.  It follows that $\zeta\geq \nu$ on the parabolic boundary of $(Q_L\setminus \overline{Q_{R/2}})\times(0,T)$.  The maximum principle then implies that $\nu \leq \zeta$ in $(Q_L\setminus \overline{Q_{R/2}}) \times (0,T)$.  Given the definition of $\zeta$ and the preliminary bound on $v^\e$ on $Q_{3R/4}$~\eqref{eq:preliminary_bound}, it follows that there exists a constant $C$ that depends only on $u_0$, $\theta_0$, $D$, $L$, $R$, and $T$ such that
\[
	v^\e \leq C
		\qquad\text{ in } Q_{L/2}(0,\theta_0)\times \left[\frac2T,\frac2T+ T\right].
\]
Since $L$ and $T$ are arbitrary, this concludes the proof.

\end{proof}

\subsection{The half-relaxed limits}

We next recall the definition of  the classical half-relaxed limits $v^*$ and $v_*$:
\begin{equation}\label{eq:half_relaxed}
	v^*(x,\theta,t) = \limsup_{\substack{(y,\eta,s) \to (x,\theta,t),\\ \epsilon \to 0}} v^\epsilon(y,\eta,s)
		\quad \text{ and } \quad
	v_*(x,\theta,t) = \liminf_{\substack{(y,\eta,s) \to (x,\theta,t),\\ \epsilon \to 0}} v^\epsilon(y,\eta,s).
\end{equation}
The existence of these  limits  is guaranteed by \Cref{lem:apriori} along with the fact that, as discussed in \Cref{sec:results}, $v^\epsilon \geq 0$.  We point out that $v_*$ is lower semi-continuous while $v^*$ is upper semi-continuous.

\subsubsection*{Equations for $v_*$ and $v^*$}

Our first step is to prove that $v_*$ and $v^*$ satisfy the limits that the theory of viscosity solutions suggest.  The issues here are  the boundary behavior and verifying the initial conditions.

\begin{lem}\label{lem:phi_*_equation}
	The relaxed lower limit  $v_*$ satisfies in the viscosity sense
	\begin{equation}\label{eq:phi_*_equation}
		\begin{cases}
			\min\left\{(v_*)_t + \overline D(\theta) |(v_*)_x|^2 + |(v_*)_\theta|^2 + 1, v_*\right\} \geq 0 \quad &\text{ in } \quad  \R \times \R^+ \times \R^+,\\
			\max\left\{  -(v_*)_\theta, \min\left\{(v_*)_t  + |(v_*)_\theta|^2 + 1, v_*\right\}\right\} \geq 0\quad & \text{ on } \quad  \R \times \{0\}\times\R^+,
		\end{cases}
	\end{equation}
and
	\begin{equation}\label{eq:lower_initial_condition}
		v_*(\cdot,\cdot ,0)
			= \begin{cases}
				0  &\text{ in  } \quad G_0,\\
				\infty  \quad &\text{ in } \quad  \overline G_0^c.
			\end{cases}
	\end{equation}
\end{lem}
\begin{proof}
We verify~\eqref{eq:phi_*_equation} first.  Assume that,  for some test function $\psi$,  $v_* - \psi$ has a strict local minimum at $(x_0,\theta_0,t_0) \in \R \times [0,\infty)\times \R^+$.  We may then choose  $\epsilon_k$  converging to $0$ and $(y_k, \eta_k, s_k)$ converging to $(x_0,\theta_0,t_0)$ as $k$ tends to infinity such that $(y_k,\eta_k,s_k)$ is a local minimum of $v^{\epsilon_k} - \psi$  in $ \R \times [0,\infty)\times [0,\infty)$ and $v_*(x_0,t_0, \theta_0) = \lim_{k\to\infty} v^{\epsilon_k}(y_k,\eta_k,s_k).$


If  $(x_0,\theta_0,t_0) \in \R \times \R^+\times \R^+$, then, for sufficiently large $k$, $(y_k,\eta_k,s_k) \in  \R \times \R^+\times \R^+$.  Since $v^\epsilon$ solves~\eqref{eq:phase}, at $(y_k,\eta_k,s_k)$, we have, at $(y_k, \eta_k, s_k)$,
\[\begin{split}
	0 \leq \psi_t - \epsilon_k \overline D^{\epsilon_k} \psi_{xx} - \epsilon_k \psi_{\theta\theta} + \overline D^{\epsilon_k} |\psi_x|^2 + |\psi_\theta|^2 + 1 - e^{-\psi/\epsilon_k}
		\leq \psi_t - \overline D |\psi_x|^2 + |\psi_\theta|^2 + 1 + o(1).
\end{split}\]
Here and in the sequel, we use $o(1)$ to mean a quantity that tends to zero in the limit.  Taking the limit as $k$ tends to infinity and using the smoothness of $\psi$ yields, at $(x_0,\theta_0,t_0)$, $0 \leq \psi_t  + \overline D \psi_x^2 + \psi_\theta^2 + 1.$
As discussed above, $v_* \geq 0$ on $\R\times \R^+\times\R^+$.  From this and the inequality above, we conclude that
\[
	\min\left\{(v_*)_t + \overline D(\theta) |(v_*)_x|^2 + |(v_*)_\theta|^2 + 1, v_*\right\} \geq 0
		\qquad \text{ in } \R\times \R^+ \times \R^+,
\]
which finishes the proof in this case.
\smallskip

Assume next that $(x_0,\theta_0,t_0) \in \R \times \{0\}\times \R^+$.  If $\eta_k > 0$ for infinitely many $k$, the fact that $v^{\epsilon_k}$ solves~\eqref{eq:phase} yields, at $(y_k,\eta_k,s_k)$,
\[
	0 \leq \psi_t - \epsilon_k \overline D^{\epsilon_k} \psi_{xx} - \epsilon_k \psi_{\theta\theta} + \overline D^{\epsilon_k} \psi_x^2 + \psi_\theta^2 + 1 - e^{-\phi^{\epsilon_k}/\epsilon_k}
		\leq \psi_t + \overline D^{\epsilon_k} \psi_x^2 + \psi_\theta^2 + 1 + o(1).
\]
Letting $k$ tend to infinity, we find, at $(x_0,0,t_0)$, $0 \leq \psi_t + \overline D \psi_x^2 + \psi_\theta^2 + 1.$
If $\eta_k = 0$ for all $k$ sufficiently large, then, since $v^{\epsilon_k}$ satisfies Neumann boundary conditions, we have  $0 \leq - \psi_\theta(y_k,0,s_k).$
Letting $k$ tend to infinity, we find $0 \leq - \psi_\theta(x_0,0,t_0)$. In either case, we have verified that
\[
	\max\left\{  -(v_*)_\theta, \min\left\{(v_*)_t  + |(v_*)_\theta|^2 + 1, v_*\right\}\right\} \geq 0\quad  \text{ on } \quad  \R \times \{0\}\times\R^+.
\]
\smallskip

Finally we need to consider the initial condition~\eqref{eq:lower_initial_condition}.  Fix $\mu >0$ and  any smooth function $\zeta \in C^\infty(\R\times [0,\infty);[0,1])$ such that $\zeta|_{\overline G_0} \equiv 0$ and $\zeta|_{\R\times \R^+ \setminus \overline G_0} > 0$.   Then
\begin{equation}\label{eq:phi_*_0}
\begin{cases}
	\max\left\{(v_*)_t + \overline D |(v_*)_x|^2 + |(v_*)_\theta|^2 + 1, v_* - \mu \zeta\right\} \geq 0 \quad 
		&\text{ in } \quad  \R \times [0,\infty)\times \{0\},\\[1mm]
	\max\left\{ -(v_*)_\theta, (v_*)_t + |(v_*)_\theta|^2 + 1, v_* - \mu \zeta\right\} \geq 0
			\quad  &\text{ in } \quad  \R \times \{0\}\times \{0\}.
\end{cases}
\end{equation}
Indeed, if $(x_0,\theta_0) \in \overline G_0$, \eqref{eq:phi_*_0} holds  since $v_* \geq 0$ and $\zeta \equiv 0$ on $\overline G_0$.  If $(x_0,\theta_0) \in \R\times \R^+ \setminus \overline G_0$ and  $v_*(0,x_0,\theta_0) < \mu \zeta(x_0,\zeta_0)$ then, since $v_*$ is finite at $(x_0,\theta_0)$, we argue exactly as in the second paragraph of this proof  to obtain $(v_*)_t + \overline D |(v_*)_x|^2 + |(v_*)_\theta|^2 + 1 \geq 0$. We proceed similarly if $\theta_0 = 0$  using the arguments of the third paragraph of this proof.  Hence, we obtain~\eqref{eq:phi_*_0}.
\smallskip

It follows immmediately from~\eqref{eq:small_on_G_0} of \Cref{lem:apriori} and the definition of $\liminf$ that $v_* = 0$ on $\{0\}\times G_0$.  If $(x_0,\theta_0) \in \R\times\R^+\setminus \overline G_0$, then we assume, by contradiction, that $v_*(x_0,\theta_0,0) < \infty$.
\smallskip

Choose $\mu$ sufficiently large so that $v_*(x_0,\theta_0,0) < \mu \zeta(x_0,\theta_0,0)$.  Let
\begin{equation}\label{eq:lambda}
	\lambda_\delta = 1 + \frac{1}{\delta} + \frac{8(1 + \overline D(\theta_0)) v_*(x_0,\theta_0,0)}{\delta}.
\end{equation}
Notice that $\lambda_\delta$ tends to infinity as $\delta$ tends to zero.  Define the test function
$\psi_\delta(x,\theta,t) := -\delta^{-1}(|x - x_0|^2 + |\theta -\theta_0|^2) - \lambda_\delta t.$
Since $v_*$ is lower semi-continuous, $v_* - \psi_\delta$ attains a  minimum at some $(x_\delta, \theta_\delta, t_\delta) \in \R \times[0,\infty)\times[0,\infty)$.  Further, $v^*(x_0,\theta_0,0) < +\infty$ and $\psi_\delta(x,\theta,t)$ tends to infinity locally uniformly away from $(x_0,\theta_0,0)$.  Thus, $(x_\delta,\theta_\delta,t_\delta)$ converges to $(x_0,\theta_0,0)$ as $\delta$ tends to zero. As $(x_\delta,\theta_\delta, t_\delta)$ is a minimum of $v_* - \psi_\delta$, we see that
\begin{equation}\label{eq:phi_*_bound}
	v_*(x_\delta, \theta_\delta, t_\delta) + \lambda_\delta t_\delta + \frac{|x_\delta-x_0|^2 + |\theta_\delta-\theta_0|^2}{\delta}
		\leq	v_*(x_0,\theta_0,0).
\end{equation}

\sloppy  We now collect four properties that hold when $\delta$ is small and rely the fact that $(x_\delta,\theta_\delta,t_\delta)$ converges to $(x_0,\theta_0,0)$ as $\delta$ tends to zero.  Firstly, by~\eqref{eq:phi_*_bound}, if $t_\delta>0$ for any $\delta$ then $v_*(x_0,\theta_0,0)>0$ and, thus, $v_*(x_\delta,\theta_\delta,t_\delta)>0$ if $\delta$ is sufficiently small due to the the lower semi-continuity of $v_*$.  Secondly,~\eqref{eq:phi_*_bound}, the lower semi-continuity of $v_*$, and the fact that $v_*(x_0,\theta_0,0) < \mu\zeta(x_0,\theta_0,0)$, imply that if $\delta$ is sufficiently small,  $0 < v_*(x_\delta,\theta_\delta,t_\delta) < \mu \zeta(x_\delta,\theta_\delta,t_\delta)$.  Thirdly, the continuity of $\overline D$ implies that $\overline D(\theta_\delta) \leq 2 \overline D(\theta_0)$ for all $\delta$ sufficiently small.  Fourthly and finally, since $\theta_0>0$, then $\theta_\delta > 0$ if $\delta$ is sufficiently small.  Fix $\delta_0>0$ such that, if $\delta \in (0,\delta_0)$ then all four properties above hold.
\smallskip

Suppose that $t_\delta>0$ for some $\delta\in(0,\delta_0)$. 
Using that $v_*$ satisfies~\eqref{eq:phi_*_equation} for $t_\delta > 0$ and $v_*(x_\delta,\theta_\delta,t_\delta)>0$, we have 
\begin{equation}\label{eq:psi_delta_inequality}
\begin{split}
	0
		&\leq \psi_t(x_\delta,\theta_\delta,t_\delta) + \overline D(\theta_\delta) \psi_x(x_\delta,\theta_\delta,t_\delta)^2 + \psi_\theta(x_\delta,\theta_\delta,t_\delta)^2 + 1\\
		&\leq -\lambda_\delta + \frac{4(2\overline D(\theta_0)+1)(|x_\delta - x_0|^2 + |\theta_\delta - \theta_0|^2)}{\delta^2} + 1.
\end{split}
\end{equation}
Above we used that $\overline D(\theta_\delta) \leq 2 \overline D(\theta_0)$.  Using now~\eqref{eq:phi_*_bound} in~\eqref{eq:psi_delta_inequality}, we find
\begin{equation}\label{eq:phi_*_contradiction}
	0 
		\leq - \lambda_\delta + \frac{4(2 \overline D(\theta_0) + 1) v_*(x_0,\theta_0,0)}{\delta} + 1.
\end{equation}
In view of the definition of $\lambda_\delta$~\eqref{eq:lambda}, the right hand side is negative.  This yields a contradiction.
\smallskip


If $t_\delta = 0$ for all $\delta\in(0,\delta_0)$, the proof is the same as above, with~\eqref{eq:phi_*_0} playing the role of~\eqref{eq:phi_*_equation}.  Indeed, as observed above, we have that $v_*(x_\delta,\theta_\delta,t_\delta) < \mu \zeta(x_\delta,\theta_\delta,t_\delta)$.  Using this and that $v_*$ satisfies~\eqref{eq:phi_*_0}, we find, at $(x_\delta,\theta_\delta,t_\delta)$, $\psi_t + \overline D |\psi_x|^2 + |\psi_\theta|^2 + 1 \geq 0.$
Using the definition of $\psi$ and the choice of $\lambda_\delta$, we obtain the same contradiction as in~\eqref{eq:phi_*_contradiction}.
\smallskip

Having reached a contradiction in both cases, we conclude that $v_*(x_0,\theta_0,0) = +\infty$.
\end{proof}

We now obtain the equation for $v^*$.  The argument  is slightly more complicated since $v^*\geq 0$ and, hence, for the first equation must consider the cases where $v^*$ is zero or positive.

\begin{lem}\label{lem:phi^*_equation}
	The upper relaxed half limit $v^*$ is a viscosity solution to 
	\begin{equation}\label{eq:phi^*_equation}
		\begin{cases}
			\min\left\{(v^*)_t + \overline D |(v^*)_x|^2 + |(v^*)_\theta|^2 + 1, v^*\right\} \leq 0 
				\quad &\text{ in } \quad  \R \times \R^+ \times \R^+,\\[1mm]
			\min\left\{  -(v^*)_\theta, \min\left\{(v^*)_t  + |(v^*)_\theta|^2 + 1, v^*\right\}\right\} \leq 0,
				\quad  &\text{ on } \quad   \R \times \{0\}\times \R^+,
		\end{cases}
	\end{equation}
	and
	\begin{equation}\label{eq:lower_initial_condition}
		v^*(\cdot, \cdot,0)
			= \begin{cases}
				0 \quad \text{ in } \quad   G_0,\\[1mm]
				\infty \quad \text{ in }  \quad \overline G_0^c.
			\end{cases}
	\end{equation}
\end{lem}
\begin{proof}
The proof of \Cref{lem:phi^*_equation} is similar to that of \Cref{lem:phi_*_equation}, thus we omit some details and provide only a sketch of the proof.
\smallskip

We first verify~\eqref{eq:phi^*_equation}.  Assume that,  for some test function $\psi$,  $v^* - \psi$ has a strict local maximum at $(x_0,\theta_0,t_0) \in \R \times [0,\infty)\times \R^+$.  We may then choose  $\epsilon_k$ converging to $0$ and $(y_k, \eta_k, s_k)$ converging to $(x_0,\theta_0,t_0)$ as $k$ tends to infinity such that $(y_k,\eta_k,s_k)$ is  a local maximum  of $v^{\epsilon_k} - \psi$ and 
\[
	v^*(x_0,t_0, \theta_0) = \lim_{k\to\infty} v^{\epsilon_k}(y_k,\eta_k,s_k).
\]
To check~\eqref{eq:phi^*_equation}, we need only consider the set $\{v^* > 0\}$ since~
\eqref{eq:phi^*_equation} is satisfied whenever $v^* = 0$. 
\smallskip

If  $t_0 > 0$ and $\theta_0 > 0$, then for sufficiently large $k$, $t_k, \theta_k >0$ and, at 
$(y_k,\eta_k,s_k)$,
\[
	0 \geq \psi_t - \epsilon_k \overline D^{\epsilon_k} \psi_{xx} - \epsilon_k \psi_{\theta\theta} + \overline D^{\epsilon_k}|\psi_x|^2 + |\psi_\theta|^2 + 1 - e^{- v^{\epsilon_k}/\epsilon_k}.
\]
Since $v^{\epsilon_k}(y_k,\eta_k,s_k)$ converges to $v^*(x_0,\theta_0,t_0) > 0$ as $k$ tends to $\infty$, the last term tends to zero as $k$ tends to infinity.  In addition, the regularity of $\psi$ implies that, after taking the limit $k$ to infinity, at $(x_0,\theta_0,t_0)$, $0 \geq \psi_t + \overline D|\psi_x|^2 + |\psi_\theta|^2 + 1.$
If $\theta_0 = 0$ we argue similarly as in \Cref{lem:phi_*_equation}.
\smallskip

We now consider the case $t_0 = 0$.  Fix any point $(x_0,\theta_0) \in G_0$.  Using~\eqref{eq:small_on_G_0}, we have that $v^\e$ converges to zero uniformly on any compact subset of $G_0\times[0,\infty)$.  Hence $v^*(x_0,\theta_0,0) = 0$.
\smallskip

\sloppy On the other hand, fix any point $(x_0,\theta_0) \in \overline G_0^c$, and notice that $v^\e(x_0,\theta_0,0) = -\e \log(u_0(x_0,\theta_0,0)) = - \e \log(0) = +\infty$.  It then follows immediately from the definition of $\limsup$ that $v^*(x_0,\theta_0,0) = \infty$.  This concludes the proof.

\end{proof}

\subsection{The equality of $v_*$ and $v^*$}\label{sec:identification}

As noted above, by construction, $v_* \leq v^*$.  In addition, $v_*$ and $v^*$ are a super- and a sub-solution to the same equation with the same initial conditions except on the small set $\partial G_0$.  In this section, we show that $v_* = v^*$.

\vspace{-.2in}
\subsubsection*{Existence and uniqueness of $I$}

\vspace{-.1in}
We outline the argument developed in Crandall, Lions and Souganidis~\cite{CrandallLionsSouganidis} that yields that there exists a unique solution to~\eqref{eq:I} with initial condition~\eqref{eq:I_0_intro}.
\smallskip

For any open, convex, $C^3$ set $U$, let  $\mathcal{C}_U: = \{ \zeta \in C^0 (\R\times [0,\infty)) : \zeta|_{U} \equiv 0\}$ and  denote by $S(t)\zeta$  the solution to~\eqref{eq:I} with the initial data $\zeta \in \mathcal{C}_U$.  The existence and uniqueness of $S(t)\zeta$ are well-understood;  see,~\cite{CrandallIshiiLions}.  In addition,  arguments as in \Cref{sec:apriori} give  bounds on $S(t)\zeta$ in $\R\times[0,\infty)\times\R^+$. 
Let $I(x,\theta,t) := \sup_{\zeta \in \mathcal{C}_{G_0}} S(t)\zeta.$
Following~\cite{CrandallLionsSouganidis}, we observe that $I$ is the unique maximal solution of~\eqref{eq:I}.  We note that, due to the Neumann boundary conditions, this does not follow directly from~\cite{CrandallLionsSouganidis}. The extension is, however, straightforward.

\vspace{-.1in}
\subsubsection*{The equality of $v_*$ and $v^*$}

\vspace{-.05in}
\begin{proof}[Proof of \Cref{prop:phi_convergence}]
First, we show that $v_* \geq I$.  To this end, fix any $\zeta \in \mathcal{C}_{G_0}$. 
Observe that $v_*(\cdot,\cdot,0) \geq \zeta$ on $\R\times[0,\infty)$.  The standard comparison principle, along with \Cref{lem:phi_*_equation}, yields $v_* \geq S(t)\zeta$ on $\R\times[0,\infty)\times \R^+$.  Since this  is true for all $\zeta$,  we find $I = \sup_{\zeta \in \mathcal{C}_{G_0}} S(t)\zeta\leq v_*.$

Next, we show that $v^* \leq I$.  Fix $\delta > 0$ and define $G_\delta := \{(x,\theta) \in G_0 : \dist((x,\theta), G_0^c) > \delta\}.$ Let $I_\delta = \sup_{\zeta\in\mathcal{C}_{G_\delta}} S(t)\zeta$.  Fix any $\sigma>0$.  By \Cref{lem:apriori}, we have that $v^*(\cdot,\cdot,\sigma)$ is finite on $\R\times [0,\infty)$ and is zero on $G_0$.  Hence, there exists $\zeta \in \mathcal{C}_{G_\delta}$ such that $v^*(\cdot,\cdot,\sigma) \leq \zeta$.  From the comparison principle, it follows that, for all $(x,\theta,t) \in \R\times[0,\infty)\times[0,\infty)$,
\[
	v^*(x,\theta,t+\sigma)
		\leq (S(t)\zeta)(x,\theta)
		\leq \sup_{\zeta'\in\mathcal{C}_{G_\delta}} (S(t)\zeta')(x,\theta)
		= I_\delta(x,\theta,t).
\]
Taking $\sigma$ to zero, we obtain $v^* \leq I_\delta$ on $\R\times[0,\infty)\times\R^+$.  Further, it is easy to see\footnote{This is intuitively clear and can be observed in many ways.  In the current manuscript, the quickest is, perhaps, using the inclusion $\{I_\delta > 0\} \subset \{w_\delta = 1\}$ seen in~\Cref{sec:inclusions}, where $w_\delta$ satisfies~\eqref{eq:w_intro} with $G_0$ replaced by $G_\delta$.  A straightforward computation using~\eqref{eq:w_variational} yields $\sigma_\delta$ such that $w_\delta(\cdot,\cdot,\sigma_\delta)|_{G_0} \equiv 0$, from which the claim follows.} that, there exists $\sigma_\delta$, which tends to zero as $\delta$ does and depends only on $G_0$ and $\delta$, such that $I_\delta(\cdot,\cdot,\sigma_\delta)\in \mathcal{C}_{G_0}$.  We conclude that
\[
	v^*(x,\theta,t+\sigma_\delta)
		\leq I_\delta(x,\theta,t+\sigma_\delta)
		\leq I(x,\theta,t)
\]
for all $(x,\theta,t) \in \R\times[0,\infty)\times\R^+$.  Taking $\delta$ to zero, we conclude that $v^* \leq I$, as desired.
\smallskip

%

Hence we have that $v_* \leq v^* \leq I \leq v_*$, which  implies  that all three functions must be equal.  In particular, we have that $v^\epsilon$ converges locally uniformly to $I$, finishing the proof.
\end{proof}

\section{The relationship between $I$, $J$, and $w$ -- \Cref{prop:geometric,prop:I_J}}\label{sec:relationship}

We now characterize the location of the front in a more tractable manner; that is we prove \Cref{prop:geometric,prop:I_J}.  We do not follow the approach of ~\cite{Freidlin1,Freidlin2}, in which the author shows directly that $I= \max\{J,0\}$ by developing a theory for and checking a condition on the minimizing paths of $J$.  As this condition is difficult to verify, we, instead, opt for a PDE proof based on the work in~\cite{MajdaSouganidis} using $w$ in an intermediate step. 
We note that, since the Hamiltonian associated to~\eqref{eq:I}, $H(x,\theta,p_x,p_\theta) := \overline D(\theta) |p_x|^2 + |p_\theta|^2 + 1$, is not homogeneous, that is, it depends on $\theta$, the arguments from~\cite{MajdaSouganidis} do not directly apply.  We outline our proof below, and make note of the differences with~\cite{MajdaSouganidis}.
\smallskip

In order to prove \Cref{prop:geometric,prop:I_J}, we show equivalence of the various level and super-level sets involved and then we apply \Cref{thm}.  The inclusion $\{J>0\}\subset\{I>0\}\subset\{w=1\}$ follows from the maximum principle, as in~\cite{MajdaSouganidis}.  To close this chain of inclusions, we require $\{w=1\}\subset\{J>0\}$.  This is accomplished in~\cite{MajdaSouganidis} via the Hopf-Lax formula; however, this only applies when the Hamiltonian is independent of $(x,\theta,t)$ and so is not useful here.  We get around this by using the fact that $w$ is given as the solution to a variational problem similar to the one defining $J$.  We can then compare these two functions directly.
\smallskip
%


In order to follow this outline, we first show the following two key facts: that $J$ is a sub-solution of~\eqref{eq:I} and that $w$ can be represented by a variational problem.

\subsection{The equation for $J$}\label{sec:J}

We first show that    $J$ solves
\begin{equation}\label{eq:J}
	\begin{cases}
		J_t + \overline D(\theta) |J_x|^2 + |J_\theta|^2 + 1 = 0 \quad &\text{ in } \quad \R \times \R^+ \times\R^+,\\[1mm]
		\min\{-J_\theta,J_t + |J_\theta|^2 + 1\} \leq 0\quad  &\text{ on } \quad \R\times \{0\}\times \R^+,
	\end{cases}
\end{equation}
from which it follows that $J$ is a sub-solution of~\eqref{eq:I}. 
The main difficulty is verifying the boundary condition.
We note that $J$ actually satisfies the Neumann boundary condition in $\theta$, but this is not necessary for our purposes so we do not show it.

\begin{proof}[Proof of~\eqref{eq:J}]



In \Cref{sec:appendix}, we discuss how the classical arguments may be easily adapted to show that $J$ solves~\eqref{eq:J} on $\R\times\R^+\times\R^+$. 
The main point is that optimal trajectories in the definition of $J$ exist and remain bounded away from the set $\R\times\{0\}$, see \Cref{sec:appendix}.  As such, one may show that the dynamic programming principle is verified and argue as usual.
\smallskip

Next, we show that $\min\left\{-J_\theta, J_t + |J_\theta|^2 + 1 \right\} \leq 0$ on $\R\times\{0\}\times\R^+$.  For any test function $\varphi$,  assume that $J-\varphi$ has a strict maximum at $(x_0,0,t_0) \in \R\times \{0\}\times \R^+$ in a ball\footnote{Here, we define a ball as follows: for any $(x,\theta,t) \in \R\times [0,\infty)\times[0,\infty)$, let $B_R(x,\theta,t) := \{(y,\eta,s) \in \R\times[0,\infty)\times [0,\infty): |x-y|^2 + |\theta-\eta|^2 + |t-s|^2 < R^2\}$.  In particular, we include only those points in the ambient space $\R\times[0,\infty)\times[0,\infty)$.} $\overline B_r(x_0,0,t_0)$. Without loss of generality, assume that  $(J-\phi)(x_0,0,t_0)=0$ and $r < t_0$. If $-\varphi_\theta(x_0,0,t_0) \leq 0$ then we are finished.    Hence, we may assume that $\varphi_\theta(x_0,0,t_0) < 0$. 
\smallskip

Fix any smooth function $\psi:[0,\infty) \to \R$ such that $\psi(0) = 0$, $\psi(1) = -1$, and $\psi(2) = 0$, which  is strictly increasing on $[0,1/2] \cup [1,\infty)$ and strictly decreasing on $[1/2,1]$. 
%
%
%
%
%
For any $\epsilon, \delta >0$, let
\[
	\varphi_{\delta,\epsilon}(x,\theta,t)
		:= \varphi(x,\theta,t) + \epsilon \psi(\theta/\delta).
\]
If $\delta \geq (2r)^{-1}$, observe that $\varphi_{\delta,\epsilon}(x,\theta,t) \geq \varphi(x,\theta,t) \geq J(x,\theta,t)$ for all $(x,\theta,t) \in \overline B_r(x_0,0,t_0)$, with equality only at $(x_0,0,t_0)$.  Define
\[
	\delta_\epsilon 
		:= \inf \big\{\delta>0 : \text{ if } \delta' > \delta \text{ then } \varphi_{\delta',\epsilon} > J \text{ on } \overline B_r(x_0,0,t_0)\setminus\{(x_0,0,t_0)\}\big\}.
\]
Then there exists $(x_\e,\theta_\e,t_\e) \in \overline B_r(x_0,0,t_0)\setminus\{(x_0,0,t_0)\}$ such that $\varphi_{\delta_\e,\e}(x_\e,\theta_\e,t_\e) = J(x_\e,\theta_\e,t_\e)$.
\smallskip

First, we claim that $\theta_\epsilon/\delta_\epsilon \in [1/2,1]$.  Since $\phi > J$ on $\overline B_r(x_0,0,t_0)\setminus\{(x_0,0,t_0\}$ and $\psi(\theta) \geq 0$ for $\theta \in (0,1/2)\cup[2,\infty)$ it cannot be that $\theta_\epsilon/\delta_\epsilon \in (0,1/2)\cup[2,\infty)$.  We now show that $\theta_\e/\delta_\e \notin(1,2)$.  We arge by contradiction, supposing that $\theta_\epsilon/\delta_\epsilon \in (1,2)$.  Let $\theta_r := \theta_\epsilon/\delta_\epsilon$. By the construction of $\psi$ there exists $\theta_l< \theta_r$ such that $\psi(\theta_l) = \psi(\theta_r)$.  Let $\delta := \delta_\epsilon \theta_r/\theta_l$.  Notice that $\delta > \delta_\e$, which implies that $\varphi_{\delta,\e} > J$ in $\overline B_r(x_0,0,t_0)\setminus\{(x_0,0,t_0)\}$ by the definition of $\delta_\e$.  Notice also that $\theta_\epsilon / \delta = \theta_l$, which implies that $\psi(\theta_\epsilon/\delta) = \psi(\theta_l) = \psi(\theta_r) = \psi(\theta_\epsilon/\delta_\epsilon)$.  Thus, we find
\[\begin{split}
	J (x_\epsilon, \theta_\epsilon,t_\epsilon)
		&< \varphi_{\delta,\epsilon}(x_\epsilon,\theta_\epsilon,t_\epsilon)
		= \varphi(x_\epsilon,\theta_\epsilon,t_\epsilon) + \epsilon \psi(\theta_\epsilon/\delta)\\
		&= \varphi(x_\epsilon,\theta_\epsilon,t_\epsilon) + \epsilon \psi(\theta_\epsilon/\delta_\epsilon)
		= \psi_{\delta_\epsilon,\epsilon}(x_\epsilon,\theta_\epsilon,t_\epsilon)
		= J(x_\epsilon,\theta_\epsilon,t_\epsilon),
\end{split}\]
which is a contradiction.  Hence, $\theta_\epsilon/\delta_\epsilon \in [1/2,1]$, and, in particular, $\psi_\theta(\theta_\epsilon/\delta_\epsilon) \leq 0$.
\smallskip

Second we claim that 
 $(x_\e, \theta_\e,t_\e)$ converges to $(x_0,0,t_0)$ as $\e$ tends to zero.  Fix any sequence $\e_k$ tending to zero as $k$ tends to zero and extract a convergent sub-sequence, which we denote the same way, such that that $\delta_{\epsilon_k}$ converges to some $\delta_0\in [0, (2r)^{-1}]$  and $(x_{\e_k},\theta_{\e_k},t_{\e_k})$ converges to some $(x_0',\theta_0',t_0') \in \overline B_r(x_0,0,t_0)$ as $k$ tends to infinity.  
 By continuity, we observe that
\[
	J(x_0',\theta_0',t_0')
		= \lim_{k\to\infty} J(x_{\e_k},\theta_{\e_k}, t_{\e_k})
		= \lim_{k\to\infty} \varphi_{\delta_{\e_k},\e_k}(x_{\e_k},\theta_{\e_k},t_{\e_k})
		= \varphi(x_0', \theta_0',t_0').
\]
It follows that $(x_0',\theta_0',t_0') = (x_0,0,t_0)$ because $J - \varphi$ is negative in $\overline B_r(x_0,0,t_0) \setminus \{(x_0,0,t_0)\}$.  
Since every sequence has a sub-sequence such that $(x_{\e_k},\theta_{\e_k},t_{\e_k})$ converges to $(x_0,0,t_0)$ 
 as $k$ tends to infinity, we conclude that $(x_\e,\theta_\e,t_\e)$ converges to $(x_0,0,t_0)$ 
 as $\e$ tends to $0$.
\smallskip


We now verify~\eqref{eq:J} on $\R\times\{0\}\times\R^+$.  Fix $\e$ sufficiently small such that 
$(x_\epsilon, \theta_\epsilon,t_\epsilon) \in \overline B_r(x_0,0,t_0)$. Notice that $\theta_\e \geq \delta_\e/2>0$, $t_\e > t_0-r>0$, which implies that $(x_\e,\theta_\e,t_\e)\in \R\times\R^+\times\R^+$.  Also, notice that $(x_\epsilon,\theta_\epsilon,t_\epsilon)$ is the location of a local maximum of $J - \varphi_{\delta_\epsilon,\e}$. 
Hence, recalling that $J$ solves~\eqref{eq:J} in $\R\times\R^+\times\R^+$, at $(x_\e,\theta_\e,t_\e)$,
\[
	0 \geq (\varphi_{\delta_\epsilon,\e})_t + \overline D|(\varphi_{\delta_\epsilon,\e})_x|^2 + |(\varphi_{\delta_\epsilon,\e})_\theta|^2 + 1
		= \varphi_t + \overline D |\varphi_x|^2 + \left|\varphi_\theta + \frac{\e}{\delta} \psi_\theta\right|^2 + 1.
\]
Because $(x_\epsilon,\theta_\epsilon,t_\epsilon)$ converges to $(x_0,0,t_0)$ as $\e$ tends to zero, $\overline D(0) = 0$, and $\varphi$ is smooth, we have
\[
	\varphi_t(x_0,0,t_0) + |\varphi_\theta(x_0,0,t_0)|^2 +  \frac{2\epsilon}{\delta_\epsilon} \varphi_\theta(x_\epsilon,\theta_\epsilon,t_\epsilon) \psi_\theta(\theta_\epsilon/\delta_\epsilon) + \frac{\epsilon^2}{\delta_\epsilon^2}|\psi_\theta(\theta_\epsilon/\delta_\epsilon)|^2 + 1\leq o(1).
\]
Recall that $\varphi_\theta(x_0,0,t_0)<0$ by assumption.  Hence,  $\varphi_\theta(x_\epsilon,\theta_\epsilon,t_\epsilon)<0$ for $\e$ sufficiently small.  Using this and the fact that that $\psi_\theta(\theta_\epsilon/\delta_\epsilon)\leq 0$, we take $\e$ to zero to find that,
\[
	\varphi_t(x_0,0,t_0) + |\varphi_\theta(x_0,0,t_0)|^2 + 1\leq 0.
\]
This concludes the proof.

\end{proof}

\subsection{A representation formula for $w$}\label{sec:geometric}

Recall that $w$ satisfies \eqref{eq:w_intro} and \eqref {eq:w_0_intro}.
Following work of Lions~\cite{Lions}, we define, for any $(x,\theta) \in \R\times\R^+$ and $p=(p_x,p_\theta) \in \R^2$, $N(x,\theta, p) 
		:= \frac{1}{2} \sqrt{ p_x^2/\overline{D}(\theta) + p_\theta^2}.$ 
Then let
\[\begin{split}
	d((x,\theta),(y,\eta)):=
		\inf_{\mathcal{A}_{x,\theta,\{(y,\eta)\},1}}\int_0^1 N(\gamma, \dot\gamma)ds.
\end{split}\]
Without the boundary, it follows from~\cite[Section 3.4]{Lions} that
\begin{equation}\label{eq:w_variational}
	w(x,\theta,t) = \inf\{ w(y,\eta,0) : d((x,\theta),(y,\eta)) \leq t\}
		= \begin{cases}
			0, \qquad &\text{ if } d((x,\theta), \overline G_0) \leq t,\\
			1, &\text{ otherwise.}
		\end{cases}
\end{equation}
The modifications in our setting are straightforward, with the main difficulties handled similarly as in our treatment of $J$.  As such, we omit it.

\subsection{The proofs of \Cref{prop:geometric} and \Cref{prop:I_J}}\label{sec:inclusions}


\begin{proof}[Proof of \Cref{prop:geometric} and \Cref{prop:I_J}]
First, we claim that $\{ I > 0 \} \subset \{w = 1\}$.  To begin, we note that $w$ is a super-solution of~\eqref{eq:I} because
\[
	2 \sqrt{\overline D(\theta) p_x^2 + p_\theta^2}
		\leq \overline D(\theta) p_x^2 + p_\theta^2 + 1.
\]
Following~\cite{MajdaSouganidis}, we let $\overline I := \tanh(I)$ and observe that $\overline I$ and $w$ satisfy the same initial data.  The maximum principle implies that $\overline I \leq w$, which, in turn, gives $\{\overline I > 0\} \subset \{w > 0\} = \{w = 1\}$.  Since $\tanh$ is increasing, we have that $\{\overline I > 0 \} = \{ I > 0\}$, and thus $\{I > 0\} \subset \{w=1\}$.
\smallskip

We note that $J$ is a sub-solution of~\eqref{eq:I} satisfying the same initial conditions as $I$.  It follows that $J \leq I$.  This implies that $\{J > 0\} \subset \{I > 0\}$.
\smallskip

Now we show that $\{w =1\} \subset \{J > 0\}$.  
  We remark that it is known that this inclusion is not true in general for propagation problems, see the appendix of~\cite{MajdaSouganidis}.  
%
%
Fix  $(x,\theta,t)\in\R\times\R^+\times\R^+$ such that $w(x,\theta,t) =1$.  It follows that $d((x,\theta), \overline{G_0}) > t$. 
Suppose that, for the sake of contradiction, $J(x,\theta,t) \leq 0$.  Let $\gamma\in\mathcal{A}_{x,\theta,G_0,1}$ be any minimizing trajectory in the formula for $J$.
Using the Cauchy-Schwarz inequality and the fact that $J(x,\theta,t) \leq 0$, we find
\[
	\sqrt t \geq \sqrt{J(x,\theta,t) + t}
		= \left(\int_0^t \left[\frac{\dot\gamma_1^2}{4\overline D(\gamma_2)} + \frac{\dot\gamma_2^2}{4}\right]ds\right)^{1/2}
		\geq \left(\int_0^t N(\gamma,\dot\gamma)ds \right) / \left(\int_0^t ds\right)^{1/2}.
\]
It follows that $\int_0^t N(\gamma,\dot\gamma)ds \leq t$.  Define the re-scaled trajectory $\tilde \gamma: [0,1]\to\R\times\R^+$ by $\tilde \gamma(s) = \gamma(st)$.  Then $\tilde \gamma(0) = (x,\theta)$ and $\tilde \gamma(1) \in \overline G_0$.  Using the definition of $d$ and then changing variables¸ yields
\[
	d((x,\theta),\overline G_0)
		\leq \int_0^1 N(\tilde \gamma, \dot{\tilde \gamma})ds
		= \int_0^t N(\gamma, \dot \gamma) ds
		\leq t.
\]
By hypothesis, $d((x,\theta),  \overline G_0) > t$, which is in contradiction to the inequality above.  It follows that $J(x,\theta,t) > 0$, and, thus, that $\{w > 0 \} \subset \{J > 0\}$.
\smallskip

Combining all inclusions above, we have that $\{J > 0\} = \{I > 0\} = \{w = 1\}$.  From \Cref{thm}, this yields the convergence of $u^\e$ to 0 in $\{w=1\}$ and  $\{J > 0\}$ in \Cref{prop:geometric} and \Cref{prop:I_J}, respectively.
\smallskip

Taking the complements of these sets and recalling that $I \geq 0$, we see that $\{J \leq 0\} = \{I = 0\} = \{w = 0\}$.  In view of \Cref{thm}, we have that $u^\e$ converges to 1 on $\Int\{w=0\}$ and $\Int \{J\leq 0\}$.  This completes the proof of \Cref{prop:geometric}.
\smallskip

To complete the proof of \Cref{prop:I_J}, we must show that $\{J < 0\} = \Int\{J \leq 0\}$.  To this end, notice that $\{J  < 0 \}$ is open, due to the continuity of $J$.  This implies that $\{J< 0 \} \subset \Int \{J \leq 0\}$.  On the other hand, fix any $(x,\theta,t) \in \Int\{J\leq 0\}$ and suppose for the sake of contradiction that $J(x,\theta,t) = 0$.  There exists $r>0$ such that $B_r(x,\theta,t) \subset \Int\{J \leq 0\}$.  It follows that $J$ has a maximum at $(x,\theta,t)$ in $B_r(x,\theta,t)$, which implies, by using the constant function $0$ as a test function, that $0 + \overline D \cdot 0^2 + 0^2 + 1 \leq 0.$ 
This is a contradiction.  Hence, $J(x,\theta,t) < 0$ and we obtain $\Int\{J\leq 0\} \subset \{J<0\}$.  We conclude that $\{J < 0\} = \Int\{J \leq 0\}$.  The proof of \Cref{prop:I_J} is now complete.
\end{proof}

\appendix
\section{Brief comments about $J$ and $w$ as a solutions of~\eqref{eq:J},~\eqref{eq:w_intro}}\label{sec:appendix}

Due to the degeneracy of \eqref{eq:J} at $\theta=0$ and the loss of coercivity of the quadratic form in the equation   as $\theta$ tends to $\infty$,~\eqref{eq:J} falls outside the classical theory of Hamilton-Jacobi equations.  In view of this, we include here some remarks that are meant to convince the reader that  $J$ and $w$ have the usual properties, that is they satisfy the dynamic programming principle, solve respectively~\eqref{eq:J} and~\eqref{eq:w_intro} in $\R\times\R^+\times\R^+$, and  their extremal paths are given by the Euler-Lagrange equations. 
Since the arguments are similar, in the remainder of the appendix we only discuss $J$.  The main observation that we establish here is that extremal paths are bounded away from $\infty$ and $0$.

\begin{lem}\label{lem:infinity}
	Suppose that  \Cref{assumption:D} and \Cref{assumption:u_0} hold, and  fix $(x,\theta,t) \in \R\times \R^+ \times\R^+$.  Let $\gamma \in H^1((0,t);\R\times\R^+)$ be a trajectory such that
	\[
			\int_0^t\left[\frac{\dot\gamma_1(s)^2}{4 \overline D(\gamma_2(s))} + \frac{\dot\gamma_2(s)^2}{4} - 1\right] ds
				\leq J(x,\theta,t) + t
	\]
	There exists $C_{x,\theta,t}$, depending only on $(x,\theta,t)$ and $\overline D$, such that, for all $s\in [0,t]$, $\gamma_2(s) \leq C_{x,\theta,t}$. 
\end{lem}
\begin{proof}
	We proceed in two steps.  First, by comparing with $\tilde \gamma$, the trajectory that connects $(x,\theta)$ and any point of $G_0$ linearly, we find $C_{x,\theta,G_0}$ depending only on $x$,$\theta$, and $G_0$, such that $J(x,\theta,t) + t \leq C_{x,\theta,G_0}t^{-1}$.
	
Secondly, we use obtain a bound on $\gamma_2$.  Indeed, for any $s\in (0,t)$, we obtain 
	\[
		\gamma_2(s) - \theta
			= \int_0^s \dot\gamma_2(r)dr
			\leq 2\sqrt{s} \sqrt{ \int_0^s \frac{\dot\gamma_2(r)^2}{4} dr}
			\leq 2 \sqrt{t} \sqrt{J(x,\theta,t)+t}
			\leq 2\sqrt{C_{x,\theta,G_0}}.
	\]
	This concludes the proof.
\end{proof}

It follows that, for any approximately extremal trajectory $\gamma$, $\gamma_2$ is bounded.  As a result, $\overline D(\gamma_2)$ is bounded from above and the quadratic form in the integrand of $J$ is uniformly coercive.  Hence any approximately extremal trajectory will be bounded in $H^1((0,t);\R\times\R^+)$.  Using compactness we obtain an extremal trajectory, $\gamma$; however, we cannot rule out the existence of times $s \in (0,t)$ such that $\gamma_2(s) = 0$.  We summarize the above observations in the following identity: let $\overline{\mathcal{A}}_{x,\theta,G_0,t} := \{\gamma \in H^1((0,t); \R\times[0,\infty)) : \gamma(0) = (x,\theta), \gamma(t) \in \overline{G}_0\}$, then
\begin{equation}\label{eq:J_bigger_set}
\begin{split}
	J(x,\theta,t) = \min_{\mathcal{A}_{x,\theta,G_0,t}} \int_0^t \left[\frac{\dot\gamma_1(s)^2}{4 \overline D(\gamma_2(s))} + \frac{\dot\gamma_2(s)^2}{4} - 1\right] ds.
\end{split}
\end{equation}
The difference between~\eqref{eq:action} and~\eqref{eq:J_bigger_set} is that, in the latter, we allow trajectories to hit the boundary $\R\times\{0\}$.  The goal of the next lemma is to rule this out.
\begin{lem}\label{lem:origin}
		Suppose that \Cref{assumption:D} and \Cref{assumption:u_0} hold. Fix $(x,\theta,t) \in \overline G_0^c\times \R^+$ and let $\gamma \in H^1((0,t);\R\times[0,\infty))$ be a trajectory such that
	\[
			J(x,\theta,t) = \int_0^t\left[\frac{\dot\gamma_1(s)^2}{4 \overline D(\gamma_2(s))} + \frac{\dot\gamma_2(s)^2}{4} - 1\right] ds.
	\]
	\begin{enumerate}[(i)]\setlength\itemsep{-5pt}
	\item For any $\alpha \in \R$, any non-empty maximal connected component of $\{\gamma_2 < \alpha\}$ includes either $0$ or $t$ as an endpoint.  In particular, $\gamma_2$ cannot have a strict local minimum.
	\item There does not exist an non-empty interval $[\underline s, \overline s] \subset [0,t]$ on which $\gamma$ is constant.
	\item Fix any $s_0 \in [0,t]$.  Then, for all $s\in (0,s_0)$, $\gamma_2(s) > \min\{\gamma_2(s_0), \theta\}$.
	\end{enumerate}
\end{lem}
\begin{proof}[Proof of (i)]
	We proceed by contradiction, assuming that there exists $s_1, s_2 \in (0,t)$ with $s_1<s_2$, $\gamma_2(s_1) = \gamma_2(s_2) = \alpha$, and $(s_1,s_2) \subset \{\gamma_2 < \alpha\}$.  We define a new trajectory $\tilde \gamma(s) = \gamma(s) \1_{[0,s_1]\cup[s_2,t]} + (\gamma_1(s),\alpha) \1_{(s_1,s_2)}$.
	It is clear that $\tilde \gamma\in\overline{\mathcal{A}}_{x,\theta,G_0,t}$. 
	By the monotonicity of $\overline D$, we see that $\overline D(\gamma_2(s)) \leq \overline D(\tilde\gamma_2(s))$ for all $s\in [0,t]$.  Thus, from~\eqref{eq:J_bigger_set}
\[\begin{split}
	J(x,\theta,t)
		&\leq \int_0^t \left[ \frac{\dot{\tilde\gamma}_1^2}{4 \overline D(\tilde\gamma_2)} + \frac{\dot{\tilde \gamma}_2^2}{4} - 1\right] ds
		= \int_{[0,s_1]\cup[s_2,t]} \left[ \frac{\dot{\gamma}_1^2}{4 \overline D(\gamma_2)} + \frac{\dot{ \gamma}_2^2}{4}\right] ds
		+ \int_{s_1}^{s_2} \frac{\dot{\gamma}_1(s)^2}{4 \overline D(\alpha)} ds + t\\
		&< \int_{[0,s_1]\cup[s_2,t]} \left[ \frac{\dot{\gamma}_1^2}{4 \overline D(\gamma_2)} + \frac{\dot{ \gamma}_2^2}{4}\right] ds
		+ \int_{s_1}^{s_2} \left[ \frac{\dot{\gamma}_1^2}{4 \overline D(\alpha)} + \frac{\dot\gamma_2^2}{4}\right] ds + t
		= J(x,\theta,t).
\end{split}\]
The strict inequality comes from the fact that $\gamma_2(s) < \alpha$ for all $s\in (s_1,s_2)$ and $\gamma_2(s_1) = \gamma_2(s_2) = \alpha$.  This is a contradiction, concluding the proof of claim (i).
\end{proof}
\begin{proof}[Proof of (ii)]
We proceed by contradiction.  Suppose that $\gamma$ is constant on $[\underline s, \overline s]$ for $0 \leq \underline s < \overline s \leq t$.  For the ease of notation, assume that $\overline s = t$, but the general case is handled similarly.  Define $\tilde \gamma(s)
		= ( \gamma_1(s\underline{s}/t)),
			 \gamma_2(s\underline{s}/t)).$
We notice that $\tilde \gamma \in \overline{\mathcal{A}}_{x,\theta,G_0,t}$. 
Thus, from~\eqref{eq:J_bigger_set},
\[\begin{split}
	J(x,&\theta,t) + t
		\leq \int_0^t \left[ \frac{\dot{\tilde \gamma}_1^2}{4\overline D(\tilde\gamma_2)} + \frac{\dot{\tilde\gamma}_2^2}{4} \right] ds
		= \left(\frac{\underline{s}}{t}\right)^2\int_0^t \left[\frac{\dot{\gamma}_1(s\underline{s}/t)^2}{4 \overline D(\gamma_2(s\underline{s}/t))} + \frac{\dot{\gamma}_2(s\underline{s}/t)^2}{4} \right] ds\\
		&= \frac{\underline{s}}{t}\int_0^{\underline s}\left[ \frac{\dot{\gamma}_1(s)^2}{4 \overline D(\gamma_2(s))} + \frac{\dot{\gamma}_2(s)^2}{4} \right] ds
		= \frac{\underline{s}}{t}\int_{0}^t\left[ \frac{\dot{\gamma}_1^2}{4 \overline D(\gamma_2)} + \frac{\dot{\gamma}_2^2}{4} \right] ds
		= \frac{\underline{s}}{t}\left( J(x,\theta,t) + t\right).
\end{split}\]
By assumption, $\underline{s} < t$.  Hence, $J(x,\theta,t) + t = 0$, which in turn implies that $\dot \gamma \equiv 0$.  This is a contradiction because $\gamma(0) \in \overline G_0^c$ and $\gamma(t) \in \overline G_0$.  This concludes the proof of claim (ii).
\end{proof}
\begin{proof}[Proof of (iii)]
We proceed by contradiction.  Suppose that there exists $s_0\in[0,t]$ and $s_1 \in (0, s_0)$ such that $\gamma_2(s_1) \leq \min\{\gamma(s_0),\theta\}$.  We assume that $\min\{\gamma(s_0),\theta\} = \gamma(s_0)$, but the argument is similar in the other case.
\smallskip

\sloppy We first consider the case when $\gamma_2(s_0) >0$.  If $\min_{s\in[0,s_0]} \gamma_2(s) < \gamma_2(s_0)$, fix any $\alpha \in (\min_{s\in[0,s_0]}\gamma_2(s), \gamma_2(s_0))$.  Applying part (i), we obtain a contradiction since $\{\gamma_2 < \alpha\}$ must have a connected component contained in $(0,s_0)$ which does not contain $0$ as an endpoint.
\smallskip

It follows that $\gamma_2(s_1) = \gamma_2(s_0)$ and that $\gamma_2(s) \geq \gamma_2(s_0)$ for all $s\in[0,s_0]$.  If $\max_{[0,s_1]} \gamma_2, \max_{[s_1,s_0]} \gamma_2 > \gamma_2(s_0)$, 
we can argue exactly as above, with the choice $\alpha \in (\gamma_2(s_0), \min\{\max_{[0,s_1]} \gamma_2, \max_{[s_1,s_0]} \gamma_2\})$, to obtain a contradiction via part (i).  Hence, we consider only the case that $\gamma_2(s) = \gamma_2(s_0)$ for all $s\in [s_1,s_0]$, though the case $\gamma_2(s) = \gamma_2(s_0)$ for all $s \in [0,s_1]$ follows similarly.   By part (ii), it must be that $\{s \in (s_1,s_0) : \dot \gamma_1(s)\neq0\}$ has positive measure.  Fix $\e>0$ to be determined, let $T_\e(s) = \e((s_2-s_1) - |2s - (s_2+s_1)|)$, and define the trajectory $\tilde \gamma(s) = \gamma(s) + (0, T_\e(s))\1_{[s_1,s_1]}(s)$. 
It is clear that $\tilde \gamma \in \mathcal{A}_{x,\theta,G_0,t}$. 
Using first that $\overline D(\theta) = \theta^p$ and a Taylor expansion and then that $\gamma_2\equiv  \gamma_2(s_0)$ in $(s_1,s_0)$, we find, from~\eqref{eq:J_bigger_set},
\[\begin{split}
	&J(x,\theta,t)+t
		\leq \int_0^t \left[ \frac{\dot{\tilde\gamma}_1(s)^2}{4 \overline D(\tilde\gamma_2(s))} + \frac{\dot{\tilde \gamma}_2(s)^2}{4}\right] ds\\
		&= \int_{s_1}^{s_0} \left[ \frac{\dot{\gamma}_1(s)^2 (1 - p\gamma_2(s_0)^{p-1} T_\e(s)) + O(\e^2))}{4 \overline D(\gamma_2(s_0))} + \e^2\right] ds + \int_{[0,t]\setminus[s_1,s_0]}  \left[ \frac{\dot{\gamma}_1(s)^2}{4 \overline D(\gamma_2(s))} + \frac{\dot{ \gamma}_2(s)^2}{4}\right] ds\\
		&= -\int_{s_1}^{s_0} \left[ \frac{\dot{\gamma}_1(s)^2(p\gamma_2(s_0)^{p-1} T_\e(s)))}{4 \overline D(\gamma_2(s_0))}  + O(\e^2)\right] ds
		 + J(x,\theta,t) + t.
\end{split}\]
Using the explicit form of $T_\e$ and that $\{s \in (0,s_1): \dot\gamma_1(s) \neq 0 \}$ has positive measure, the first term on the last line is negative when $\e$ is sufficiently small.  The above then simplifies to $J(x,\theta,t) < J(x,\theta,t)$, which is a contradiction.
\smallskip

Under the assumption that $\gamma_2(s_0) >0$, we have examined all cases and obtained a contradiction in each one.  We conclude that, when $\gamma_2(s_0)>0$, the claim holds.
\smallskip

Suppose that $\gamma_2(s_0) = 0$.  
By applying part (i) with $\alpha$ tending to zero, we find $\gamma_2(s) = 0$ for all $s \in I_{s_1}$, where $I_{s_1}$ is either $[0,s_1]$ or $[s_1,t]$.  Since $\overline D(\gamma_2(s)) = 0$ for all $s \in I_{s_1}$, it follows that $\dot\gamma_2(s) = 0$ for all $s \in I_{s_1}$, otherwise $J$ would be infinite.  Thus, $\gamma$ is constant on $I_{s_1}$, which contradicts part (ii).   This concludes the proof.
\end{proof}

Since extremal trajectories remain bounded away from zero, they do not ``see'' the boundary.  Hence  the standard theory of Hamilton-Jacobi equations applies  showing that $J$ solves~\eqref{eq:J} and has all the expected properties.


\section{The precise location of the front}\label{sec:front_location}

What follows is a somewhat informal discussion of how to compute and prove the precise asymptotics of the front location in~\eqref{eq:cane_toads_unscaled} when the initial data has compact support.  We first discuss how to ``guess'' the asymptotics in terms of an abstract representation formula using the limiting equation~\eqref{eq:I}.  Second, we outline the main modifications to the work in~\cite{BHR_acceleration} in order to prove this abstract guess. Finally, we compute an explicit value for this guess from the abstract formula.  The work below is not rigorous, but it is a simple exercise to turn this discussion into a proof.

\subsubsection*{Connecting the front location with the Hamilton-Jacobi equation}

We make precise what we mean by ``front'' in this context.  For a solution $u$ of~\eqref{eq:cane_toads_unscaled}, we refer to the region where $x>0$ and $\max_\theta u(x,\theta,t)$ transitions from $1$ to $0$ as the front, see \Cref{fig:front_location}.  As we shall see, up to lower order terms, it is enough to fix any $m\in (0,1)$ and track the level set of $u$ of height $m$; that is we may define the front as $\max \{ x: \exists \theta>0, u(x,\theta,t) = m\}$, cf.~\cite[Section 1]{BHR_acceleration}.
\smallskip

We discuss, heuristically, that the front location corresponds to the location of the boundary of the zero level set of $I$ when $G_0 = \{(0,0)\}$.  We do this by noting of $\{J = 0\} = \partial \{I = 0\}$ (see \Cref{sec:relationship}), and using $J$ for all computations.  Due to the fact that the assumption $G_0 = \{(0,0)\}$ falls outside the scope of this paper (cf.~\Cref{assumption:u_0}), all discussion in this subsection is not rigorous; however, we discuss below how to make it rigorous (see the next subsection).
\smallskip

Roughly, to see the connection between the solution $u$ of~\eqref{eq:cane_toads_unscaled} with the function $J$ of~\eqref{eq:action} we proceed as follows.  Fix $t>0$ and let $\e_t = t^{-1}$.  For any $(x,\theta,s) \in \R\times \R^+ \times [0,\infty)$, define
\[
	u^{\e_t}(x,\theta,s) = u\left(\frac{x\sqrt{D(1/\e_t)}}{\e_t}, \frac{\theta}{\e_t}, \frac{s}{\e_t }\right)
		\quad\text{ and } \quad
	v^{\e_t} = -\e_t \log u^{\e_t}.
\]
From \Cref{prop:I_J}, we expect that, as $t$ tends to infinity, and thus $\e_t$ tends to zero,
\begin{equation}\label{eq:u_to_J}
	u(x t \sqrt{D(t)}, \theta t, st)
		= u^{\e_t}(x,\theta, s)
		\to \begin{cases}
			1, \qquad &\text{ if } J(x,\theta,s) < 0,\\
			0, &\text{ if } J(x,\theta,s) > 0,
		\end{cases}
\end{equation}
and that $v^{\e_t}$ converges to $J$, where $J$ is given by~\eqref{eq:action} with the set $G_0$ to be determined.  Since $u(\cdot,\cdot,0)$ has compact support, it follows that $u^{\e_t}(\cdot,\cdot,0)$ tends to zero locally uniformly on $\{(0,0)\}^c$ and that $u^{\e_t}(0,0,0)$ is positive.  Heuristically, we then expect $v^{\e_t}(\cdot,\cdot,0)$ to converge to $0$ on $\{(0,0)\}$ and $\infty$ on $\{(0,0)\}^c$.  This, in view of the convergence of $v^{\e_t}$ to  $J$, suggests that, in the definition of $J$~\eqref{eq:action}, we should take $G_0 = \{(0,0)\}$.
\smallskip

Let
\begin{equation}\label{eq:x_f}
	x_f
		:= \max \left\{ x: \exists \theta > 0, 
 J(x,\theta,1) = 0\right\}
 		= \max \left\{ x : \min_\theta J(x,\theta,1) = 0\right\}.
\end{equation}
The second equality above is easy to check by hand.  It is also easy to observe that $|x| < x_f$ implies that $\max_\theta J(x,\theta,1) = 1$ and $|x| > x_f$ implies that $\max_\theta J(x,\theta,1) < 0$.  Returning to~\eqref{eq:u_to_J} and setting $s=1$, we expect that as $t$ tends to $\infty$, if $|x| > x_f$ then $u(xt \sqrt{D(t)}, \theta t, t)$ converges to $0$, while if $|x| < x_f$ then $u(xt \sqrt{D(t)}, \theta t, t)$ converges to $1$ for some $\theta$.  This suggests that the front location is given by
\begin{equation}\label{eq:front_location}
	x_f t \sqrt{ D(t)} + o( t \sqrt{D(t)}).
\end{equation}

\subsubsection*{How to make this rigorous}

There are two approaches that one may use to make~\eqref{eq:front_location} rigorous.  The first is to develop the theory of ``maximal solutions'' to accomodate cases such as ours, where $G_0$ is not a smooth open set, but, instead, a one-point set.  The second is to re-use the approach of~\cite{BHR_acceleration}, in our context.  For simplicity, we discuss the second approach now.
\smallskip

A slightly stronger assumption than \Cref{assumption:D} is that there exists a $C^1$ function $F: \R^+ \to \R^+$ and a real number $p>0$ such that $D(\theta)/F(\theta)$ converges to $1$ 
and $\theta \partial_\theta \log F$ converges to $p$ as $\theta $ tends to $\infty$.
This is satisfied by the example~\eqref{eq:D_example}, with the choice $F(\theta) = \theta \log(\theta+1)$.  Under this hypothesis, the strategy from~\cite{BHR_acceleration}, may be repeated with the following minimal adaptation.
\smallskip

We first discuss the proof that the front location is bounded below by~\eqref{eq:front_location}.  Let $x_f$ be as above and define $\theta_f$ to be a point such that $J(x_f,\theta_f,1) = \min_\theta J(x_f,\theta,1)=0$.  The lower bound in~\cite{BHR_acceleration} is obtained by building a sub-solution along moving, growing ellipses.  The major difficulty in adapting this strategy in our context is identifying the correct trajectory for the ellipse to follow.  Let $(X,\Theta)$ be the optimal trajectory in the definition of $J$~\eqref{eq:action} beginning at $(X(0),\Theta(0)) = (0,0)$ and ending at $(X(1),\Theta(1)) = (x_f,\theta_f)$. 
Instead of using the trajectories in~\cite[equation~(4.9)]{BHR_acceleration}, we define, for any large time $T>0$, and use the trajectory $\gamma_{T}(t) = (X_{T}(t), \Theta_{T}(t))$ where $(X_{T},\Theta_{T})$ satisfy
\[
	\Theta_{T}(t) = T \Theta(t/T),
		\quad\text{ and } \quad
	X_{T}(t) = \int_0^t \sqrt{\frac{F(\Theta_{T}(s))}{\Theta(s/T)^p}} \dot X(s/T) ds.
\]
From our assumptions, we see that $F(\Theta_{T}) \sim D(\Theta_T) \sim \Theta^p D(T)$, so that $X_{T}(T) \sim x_f T \sqrt{D(T)}$.  In the case considered in~\cite{BHR_acceleration}, $F(\theta) = \theta$ and the trajectories above are exactly those used in~\cite{BHR_acceleration}.
\smallskip

Once the trajectories have been determined, one may complete the proof of the lower bound exactly as in~\cite{BHR_acceleration} with all further modifications straightforward.  The reason for the additional assumptions on the regularity of $F$ can be seen in the hypotheses of~\cite[Lemma~4.1]{BHR_acceleration}.  This yields, for all $m \in (0,1)$,
\[
	\liminf_{t\to\infty} \frac{\max\{x \in \R : \exists \theta>0, u(x,\theta,t) \geq m\}}{t \sqrt{D(t)}} \geq x_f.
\]
\smallskip

On the other hand, an upper bound may be easily obtained using the Hamilton-Jacobi set-up (see, for a similar argument,~\cite{CalvezHendersonMirrahimiTuranova}).  We note that the explicit upper bound in~\cite{BHR_acceleration} cannot be used here since it is a particularity of the case $D(\theta) = \theta$.  We conclude that, for all $\overline x > x_f$.
\[
	\lim_{t\to\infty} \sup_{x \geq \overline x t \sqrt{D(t)}, \theta > 0} u(x, \theta, t) = 0.
\]
In particular, this implies that
\[
	\lim_{t\to\infty} \frac{\max\{x \in \R : \exists \theta>0, u(x,\theta,t) = m\}}{t \sqrt{D(t)}} = x_f.
\]
and we conclude that the front location is given by~\eqref{eq:front_location}, as claimed.

\subsubsection*{Computing $x_f$ using~\eqref{eq:x_f}}

We now compute $x_f$ explicitly.  We have the Hamiltonian system for $(X,\Theta, P, Q)$:
\begin{equation}\label{eq:Hamiltonian_system}
	\dot X = 2 P \overline D, \quad
	\dot \Theta = 2 Q,
		\qquad\text{ and }\qquad
	\dot P = 0, \quad
	\dot Q = - \overline D' P^2,
\end{equation}
with the boundary conditions $(X(0),\Theta(0)) = (0,0)$ and $(X(1), \Theta(1)) = (x_f,\theta_f)$.  We see that $P = A/2$ for some constant $A$ depending on $(x_f,\theta_f)$.  Next, we differentiate the equation for $\Theta$ to get that $\ddot \Theta = 2 \dot Q = - \frac{A^2}{2} \overline D'(\Theta).$ 
Multiplying this by $\dot \Theta$ and integrating, we find
\begin{equation}\label{eq:dot_Theta^2}
	\dot\Theta(s)^2 = \dot\Theta(0)^2 - A^2 \overline D(\Theta(s)).
\end{equation}
Further, from~\eqref{eq:Hamiltonian_system}, we see that $\dot X = A \overline D(\Theta)$.  Hence, we have that
\[
	0 = J(x_f,\theta_f,1)
		= \int_0^1 \left( \frac{\dot X^2}{4\overline D(\Theta)} + \frac{\dot \Theta^2}{4} - 1 \right) ds
		= \int_0^1 \left(\frac{A^2 \overline D(\Theta)}{4} + \frac{\dot\Theta(0)^2 - A^2 \overline D(\Theta)}{4}  -1 \right) ds,
\]
which implies that $\dot\Theta(0) = 2$.
\smallskip

To compute $A$, we need the following key observation:
\begin{lemma}\label{lem:increasing}
	Suppose that $(X,\Theta)$ is the minimizing trajectory given above.   Then $\dot \Theta(s) > 0$ for all $s \in [0,1)$ and $\dot \Theta(1) = 0$.
\end{lemma}
Heuristically this is because any downward motion in $\Theta$ could be used instead to increase $X$ for the same ``cost.''  The details of the proof can be seen in~\cite{CalvezHendersonMirrahimiTuranova}.
From \Cref{lem:increasing} and~\eqref{eq:dot_Theta^2}, we find $	\dot\Theta(s)(4 - A^2 \overline D(\Theta))^{-1/2} = 1.$
Using that $\overline D(\theta) = \theta^p$ and integrating, we see that
\begin{equation}\label{eq:theta_hypergeometric}
	\Theta(s) = \frac{1}{\mu} F_p^{-1}(2 \mu s),
\end{equation}
where $\mu = 2^{-2/p} A^{2/p}$ and $F_p(s) = \int_0^s (1 - r^p)^{-1/2} dr$.

Using~\eqref{eq:theta_hypergeometric} along with the condition $\dot\Theta(1) = 0$, we see that $0 = \dot \Theta(1)^2 = 4 - A^2 \mu^{-p} (F_p^{-1}(2\mu))^p$.
Re-arranging this and using the formula for $\mu$, this yields $1 = F_p^{-1}(2\mu)$, which can be re-written $2\mu = F_p(1)$.
To compute $F_p(1)$, we use the easy-to-establish identities
\begin{equation}\label{eq:integral1}
	\int_0^1 \frac{r^p}{\sqrt{1-r^p}}dr
		= \frac{2}{p} \int_0^1 \sqrt{1 - r^p}\ dr,
\end{equation}
and
\begin{equation}\label{eq:integral2}
	\int_0^1 \sqrt{1 - r^p}\ dr
		= \frac{1}{p}\int_0^1 r^{\frac1p -1}\sqrt{1-r}\ dr
		= \frac{1}{p} \beta\left(\frac1p,\frac32\right),
\end{equation}
where $\beta$ is the beta function, see~\cite[Section 6.2]{AbramowitzStegun}.  Using~\eqref{eq:integral1} and~\eqref{eq:integral2} yields
\begin{equation}\label{eq:F_p}
\begin{split}
	F_p(1)
		&= \int_0^1 \frac{1 - r^p + r^p}{\sqrt{1-r^p}} dr
		= \frac{p + 2}{p} \int_0^1 \sqrt{1- r^p}\ dr
		= \frac{p+2}{p^2} \beta\left(\frac1p,\frac32\right).
\end{split}
\end{equation}
Thus, we have a formula for $\mu$, which, in turn, yields a formula for $A = 2\mu^{p/2}$.

Having computed $A$, we are now in a position to conclude.  Indeed, $\dot X
		= A \overline D(\Theta)
		= \frac{A}{\mu^p} \left( F_p^{-1}(2\mu s)\right)^p
		= \frac{4}{A} \left( F_p^{-1}(2\mu s)\right)^p.$
Using~\eqref{eq:integral1} and~\eqref{eq:integral2}, we find
\[\begin{split}
	x_f &= X(1)
		= \frac{4}{A}\int_0^1 F_p^{-1}(2\mu s)^p ds
		= \frac{4}{AF_p(1)} \int_0^1 \frac{r^p}{\sqrt{1 - r^p}} dr
		= \frac{4}{A F_p(1)} \frac{2}{p^2} \beta\left(\frac1p,\frac32\right)
		= \frac{8}{A}\frac{1}{p+2}.
\end{split}
\]
The third equality comes from the change of variables $r = F_p^{-1}(2\mu s)$.

Plugging in for $A$, we have
\begin{equation}\label{eq:c316}
	x_f
		= \frac{8}{2+p}\frac{1}{A}
		= \frac{8}{2+p} \left(2^{1- p/2} \frac{p+2}{p^2}\beta\left(\frac1p,\frac32\right)\right)^{-1}.
\end{equation}
We use the well-known identities $\beta(x,y) = \Gamma(x)\Gamma(y)/\Gamma(x+y)$, $\Gamma(x+1) = x\Gamma(x)$, and $\Gamma(1/2) = \sqrt \pi$ (see~\cite[Section 6]{AbramowitzStegun}), to simplify the expression involving $\beta$.  Indeed,
\[\begin{split}
	\frac{p + 2}{p^2} \beta\left(\frac1p,\frac32\right)
		&= \frac{p+2}{p^2} \frac{\Gamma\left(\frac32\right)  \Gamma\left(\frac1p\right)}{\Gamma\left(\frac32 + \frac1p\right)}
		= \frac{p+2}{2p} \frac{\Gamma\left(\frac12\right)  \Gamma\left(1+\frac1p\right)}{\left(\frac12 + \frac1p\right)\Gamma\left(\frac12 + \frac1p\right)}
		= \sqrt \pi \frac{\Gamma\left( 1 + \frac1p\right)}{\Gamma\left(\frac12 + \frac1p\right)}.
\end{split}\]
Combining this with~\eqref{eq:c316}, yields the desired expression~\eqref{eq:front_location_formula}.

%

\end{document}